\RequirePackage{fix-cm}
\documentclass{svjour3}                     
\smartqed  
\usepackage{graphicx}
\usepackage{amsmath}
\usepackage{amssymb}
\usepackage{graphicx}
\usepackage{setspace}
\usepackage{xcolor}
\usepackage{booktabs}
\usepackage{url}            
\usepackage{multirow}
\usepackage[caption=false]{subfig}

\begin{document}

\title{No-collision Transportation Maps
}


\author{Levon Nurbekyan\thanks{L.N. was supported by Simons Foundation and the Centre de Recherches Math\'{e}matiques, through the Simons-CRM scholar-in-residence program, AFOSR MURI FA9550-18-1-0502, and ONR grant  N00014-18-1-2527.}  \and Alexander Iannantuono \and  Adam M. Oberman\thanks{This material is based on work supported by the Air Force Office of Scientific Research under award number FA9550-18-1-0167 and by a grant from the Innovative Ideas Program of the Healthy Brains and Healthy Lives initiative (HBHL), McGill University
}}


\institute{L. Nurbekyan \at
              Department of Mathematics, UCLA\\
              \email{lnurbek@math.ucla.edu}           
           \and
           A. Iannantuono\at
           Department of Mathematics and Statistics, McGill\\
           \email{alexander.iannantuono@mail.mcgill.ca}
           \and
           A. M. Oberman \at
           Department of Mathematics and Statistics, McGill\\
           \email{adam.oberman@mcgill.ca}
           }

\date{Received: date / Accepted: date}

\maketitle

\begin{abstract}
Transportation maps between probability measures are critical objects in numerous areas of mathematics and applications such as PDE, fluid mechanics, geometry, machine learning, computer science, and economics. Given a pair of source and target measures, one searches for a map that has suitable properties and transports the source measure to the target one. Here, we study maps that possess the \textit{no-collision} property; that is, particles simultaneously traveling from sources to targets in a unit time with uniform velocities do not collide. These maps are particularly relevant for applications in swarm control problems. We characterize these no-collision maps in terms of \textit{half-space preserving} property and establish a direct connection between these maps and \textit{binary-space-partitioning (BSP) tree} structures. Based on this characterization, we provide explicit BSP algorithms, of cost $O(n \log n)$, to construct no-collision maps.  Moreover, interpreting these maps as approximations of optimal transportation maps, we find that they succeed in computing nearly optimal maps for $q$-Wasserstein metric ($q=1,2$).  In some cases, our maps yield costs that are just a few percent off from being optimal.

\end{abstract}

\section{Introduction}

Given a probability measure $\mu\in \mathcal{P}(\mathbb{R}^d)$ and a Borel measurable map $T:\mathbb{R}^d \to \mathbb{R}^d$ the push-forward of $\mu$ by $T$ is a probability measure $T\sharp \mu$ that is given by
\begin{equation*}
T\sharp\mu(B)=\mu(T^{-1}(B)),~\forall B~\mbox{Borel}.
\end{equation*}
Furthermore, for probability measures $\mu,\nu \in \mathcal{P}(\mathbb{R}^d)$ and a Borel measurable mapping $T:\mathbb{R}^d \to \mathbb{R}^d$ we say that $T$ transports $\mu$ to $\nu$ if
\begin{equation}\label{eq:nu=Tsharpmu}
\nu = T \sharp\mu.
\end{equation}
Heuristically, $T\sharp \mu$ is the measure obtained from $\mu$ by a rearrangement that sends (transports) a point $x$ to $T(x)$. Therefore, one can think of \eqref{eq:nu=Tsharpmu} as a measure respecting correspondence of $\mu$ and $\nu$ where a point $x$ in the support of $\mu$ is being matched with a point $T(x)$ in the support of $\nu$. This perspective readily yields numerous applications of finding a suitable $T$ for given $\mu,\nu$. However, $T$ may not exist in general and, if exists, may not be unique. Therefore, a natural problem is to search for $T$ with some useful properties.

\textit{Optimal transportation} is a remarkably rich theory that addresses this problem \cite{villani'03,villani'09}. In this theory, one associates a cost, $c(x,y)$, for transporting a unit of mass from point $x$ to $y$. Then, one searches for a transportation (correspondence) that minimizes the cost of transporting $\mu$ into $\nu$; that is,
\begin{equation}\label{eq:ot_monge}
\min_{\nu=T\sharp \mu} \int_{\mathbb{R}^d}c(x,T(x))d\mu(x).
\end{equation}

This setup of the transportation problem turns out to be incredibly fruitful far beyond the question of finding a $T$ such that \eqref{eq:nu=Tsharpmu} holds. Indeed, applications of optimal transportation theory include PDE, fluid mechanics, geometry, machine learning, computer science, economics, and so on \cite{villani'03,villani'09,peyre'19}. However, together with extremely useful theoretical properties \eqref{eq:ot_monge} comes with a considerable computational cost. Hence, development of fast algorithms for solving or estimating \eqref{eq:ot_monge} is vital for the applications.

In this note, we are motivated by particular applications of the transportation problem \eqref{eq:nu=Tsharpmu} in swarm control. Specifically, we aim at finding $T$ that have \textit{no-collision} property; that is, for every $x_1\neq x_2 \in \mathrm{supp}(\mu)$ one has that $(1-\lambda)x_1+ \lambda T(x_1)\neq (1-\lambda)x_2+ \lambda T(x_2)$ for all $\lambda\in (0,1)$. In other words, if we transport mass (particles) along straight lines with constant speeds in a unit time, the point-masses will not collide on the way. Recall that $\left((1-\lambda)\mathrm{Id}+\lambda T\right) \sharp \mu $ is the \textit{displacement interpolation} measure between $\mu$ and $\nu$ that is often used as a shape combination or interpolation technique in imaging. Thus, no-collision property can be interpreted as \textit{no information loss} in that context.

Optimal transportation maps corresponding to $c(x,y)=h(x-y)$ with a strictly convex and even $h$ do have no-collision property \cite{villani'03,villani'09,ambrosio'08}. However, these maps are expensive to calculate and normally require global optimization. The Hungarian algorithm for assignment problem \cite{schrijver2003combinatorial} is $O(n^3)$. The transportation plan computed by linear programming is potentially faster, but still requires optimization. The entropic regularization  \cite{cuturi2013sinkhorn} is explicit and fast, but has an additional regularization parameter which affects both the quality (plan versus map) and the cost of the solution.

To the best of our knowledge, the only algorithms that achieve remarkable near-linear complexity are the ones in \cite{altschuler2018approximating} (see also \cite{cuturi2013sinkhorn,cuturi'14,Genevay2016,Altschuler2017,peyre'19,schmitzer2019stabilized} for preceding work) and \cite{matt'18}. The algorithm in \cite{altschuler2018approximating} achieves a complexity $\tilde{O}\left(\frac{n}{\epsilon^3}\left(\frac{C\log n}{\epsilon} \right)^d\right)$ for $c(x,y)=|x-y|^2$ , where $\epsilon$ is the additive approximation error for the optimal cost. Bases of the algorithm are the entropic regularization of the optimal transportation problem \cite{cuturi2013sinkhorn} and low rank approximations of the Gaussian kernel matrices. However, the algorithm returns an \textit{approximate transportation plan} instead of a map. Additionally, this plan is in an \textit{implicit factored form}. Thus, calculation of the corresponding transportation map will impose quadratic complexity. Finally, the resulting map may not have the no-collision property due to an approximation error.

In \cite{matt'18}, the authors achieve a complexity $O(n\log n)$ based on a clever back-and-forth $H^1$ gradient ascent scheme on the dual formulation of \eqref{eq:ot_monge} and a fast calculation of $c$-transforms (generalized Legendre transforms) for $c(x,y)=\sum_{i=1}^d h_i(x_i-y_i)$. However, the algorithm requires a regular grid discretization.

We propose an explicit binary-space-partitioning (BSP) algorithm that constructs no-collision transportation maps in $O(n \log n)$ time in full generality without any optimization. We achieve this goal in two steps. First, we provide a complete characterization of no-collision maps in terms of \textit{half-space preserving} property. More precisely, in Theorem \ref{thm:no_collision_char}, we show that a map has no-collision property iff for every pair of points in the source there exists a hyperplane that separates these points together with a parallel hyperplane separating the images of these points. In particular, any map that respects BSP tree structures on reference and target measures induced by the same partitioning directions is a no-collision map. This is a critical connection between BSP trees and no-collision maps that was not observed previously.

Second, we build transportation maps induced by BSP tree structures on reference and target measures.
More precisely, given $\mu,\nu$, we partition $\mathrm{supp}(\mu),~\mathrm{supp}(\nu)$ and resulting subsets by a sequence of hyperplanes so that each resulting subset has equal masses on two sides of the partitioning hyperplane. Furthermore, to every point in $\mathrm{supp}(\mu)$ and $\mathrm{supp}(\nu)$, we assign a binary sequence that records in which side of partitioning hyperplanes the point fell during the construction. These binary sequences correspond to BSP tree structures on $\mathrm{supp}(\mu),~\mathrm{supp}(\nu)$ where nodes correspond to partitioning hyperplanes. Finally, the transportation map matches the leaves in $\mathrm{supp}(\mu)$ and $\mathrm{supp}(\nu)$.

Our analysis yields a flexible construction method since we allow for an arbitrary sequence of partitioning directions as long as constituent subsets shrink in diameter. A natural choice corresponds to successive partitions by hyperplanes orthogonal to elements of a basis. In Theorem \ref{thm:half_space_pres_map}, we prove that such partitions induce a well-defined transportation map. Moreover, in Theorem \ref{thm:hat(t)_cty}, we show that these maps are a.e. continuous.

Our method for constructing transportation maps depends on decomposing each measure recursively, using the half-spaces with the same direction vectors at each step of the recursion. In this case, the construction does not involve optimization: only a median search. Therefore, our run-time is $O(n\log n)$ by using a $O(n)$ median-search algorithm \cite{blum'73,musser'97}.

This method can also be used to decompose a single measure, that corresponds to a $k-d$ tree decompositions. These are particular instances of BSP trees and are used for nearest neighbor search~\cite{indyk2004nearest}.  Typically, the setting is low dimensional, and the direction vectors are standard basis vectors.  An important recent application of nearest neighbor search in high dimensions is  image retrieval \cite{johnson2019billion}.   In this case, modern methods use a deep neural network to  output a feature vector from an image. The cosine of the angle between feature vectors is used as a distance.  Given a new image, similar images are found by nearest neighbor search using the distance on the feature vectors.    However, $k-d$ trees are less effective for high dimensional data, since nearest neighbors can lie on different sides of the hyperplanes used to bisect the data.  An alternative in high dimensions is to use randomized vectors \cite{dasgupta2008random,dasgupta2013randomized}.  A further improvement is high dimensions is to use multiple randomized $k-d$ trees, \cite{li2017fast}, since multiple tree reduce the possibility of an unfortunate splitting.

BSP and $k-d$ trees are ubiquitous tools in computer science and image processing \cite{indyk2004nearest,radha96,Wu92}. Nevertheless, to the best of our knowledge, current work is the first one to establish a direct connection between these structures and efficient constructions of transportation maps.

One can also optimize for the partitioning directions. For instance, one can search for directions that minimize transportation cost. Here, we do not address this problem which is a subject of our future research.

Given partitioning directions, BSP trees induce a total order on $\mathrm{supp}(\mu)$ and $\mathrm{supp}(\nu)$. Furthermore, the no-collision transportation map corresponding to this partitioning sequence is the unique \textit{order-preserving} or \textit{monotone} map between $\mathrm{supp}(\mu)$ and $\mathrm{supp}(\nu)$ that respects this order.

Maps that we obtain have several critical properties. Firstly, these are no-collision maps; that is, point-masses do not collide when we perform the transportation dynamically in a unit time. Secondly, these maps are scale-consistent. More precisely, the map on a coarser scale is the concatenation of maps on a finer scale: Proposition \ref{prp:synthesis}. Thirdly, these are transitive maps; that is, for a fixed sequence of partitioning directions, the map from $\mu$ to $\nu$ is the composition of the maps from $\mu$ to $\rho$ and $\rho$ to $\nu$ for any $\mu,\nu,\rho$: Proposition \ref{prp:transitivity}. The last two properties yield that the calculations of these maps are amenable to decentralization and parallelization techniques that can further enhance the speed and robustness of the calculations.

Since no-collision maps are closely related to optimal transportation maps, a natural question is to test the performance of these maps as estimators for the optimal maps. We perform tests in the two-dimensional case and successively choose horizontal and vertical partitioning directions. Interestingly, we find that these maps succeed in achieving nearly optimal costs for $c(x,y)=\|x-y\|^q_p$ for $p=1,2,\infty$ and $q=1,2$.

In our future research, we plan to investigate further these approximation results. Particularly interesting are the questions related to the characterization of distortion and approximation errors of these no-collision maps. Recall that approximate transportation maps are critical in creating nearest neighbor databases based on optimal transportation distances \cite{Indyk:2003,Indyk:2004,Indyk:2007}.

The paper is organized as follows. In Section \ref{sec:hsp_maps}, we fully characterize no-collision transportation maps in terms of the half-space-preserving property: Theorem \ref{thm:no_collision_char}. Furthermore, we prove our main existence and regularity results for half-space preserving maps, Theorems \ref{thm:half_space_pres_map} and \ref{thm:hat(t)_cty}. Next, in Section \ref{sec:hsp_maps_further}, we discuss further critical properties of half-space-preserving transportation maps. In Section \ref{sec:numerics}, we describe the implementation of our algorithm and present the numerical experiments. Finally, in the Appendix, we present the proofs of Theorems \ref{thm:half_space_pres_map} and \ref{thm:hat(t)_cty}.

\section{Half-space-preserving transportation maps}\label{sec:hsp_maps}

Here, we characterize the no-collision maps in terms of a half-space-preserving property. Based on this characterization, we introduce a new class of transportation maps that have the no-collision and other intriguing properties. In particular, the construction of these maps is very quick because it requires only finding medians.
\begin{definition}\label{eq:halfspace_preserving}
	A map $T:\Omega \subset \mathbb{R}^d\to \mathbb{R}^d$ is half-space-preserving if for every $x_1\neq x_2 \in \Omega$ there exists a direction $v \in \mathbb{S}^{d-1}$ such that
	\begin{equation*}
	x_2\cdot v-x_1 \cdot v\geq 0,\quad T(x_2)\cdot v-T(x_1) \cdot v\geq 0,
	\end{equation*}
	and at least one of the inequalities is strict.
\end{definition}
Geometrically, $T$ is half-space-preserving if for every $x_1,x_2$ there exists a pair of parallel hyperplanes that separate $x_1$ from $x_2$ and $T(x_1)$ from $T(x_2)$ in such a way that $x_i$ and $T(x_i)$ are on the same side of the corresponding hyperplane for $i=1,2$. It turns out that the half-space-preserving transportation maps are precisely the ones with the no-collision property.
\begin{theorem}\label{thm:no_collision_char}
	A map $T:\Omega \subset \mathbb{R}^d\to \mathbb{R}^d$ is half-space-preserving if and only if it has the no-collision property.
\end{theorem}
\begin{proof}Assume that $T$ is a half-space-preserving map, and $x_1\neq x_2 \in \Omega$. Thus, there exists $v\in \mathbb{S}^{d-1}$ is such that
	\begin{equation*}
	x_2\cdot v-x_1 \cdot v\geq 0,\quad T(x_2)\cdot v-T(x_1) \cdot v\geq 0,
	\end{equation*}
	and at least one of these inequalities is strict. Furthermore, assume by contradiction that there exists $\lambda \in (0,1)$ such that
	\begin{equation*}
	(1-\lambda)(x_1-x_2)+\lambda (T(x_1)-T(x_2)) =0,
	\end{equation*}
	or equivalently
	\begin{equation*}
	T(x_1)-T(x_2)= \frac{\lambda}{1-\lambda} (x_2-x_1).
	\end{equation*}
	But then we have that
	\begin{equation*}
	\begin{split}
	T(x_2)\cdot v-T(x_1) \cdot v=-\frac{\lambda}{1-\lambda}(x_2\cdot v -x_1\cdot v)\leq 0,
	\end{split}
	\end{equation*}
	and so
	\begin{equation*}
	x_2\cdot v-x_1 \cdot v= T(x_2)\cdot v-T(x_1) \cdot v = 0,
	\end{equation*}
	that is a contradiction. Therefore, half-space-preserving maps possess the no-collision property.

	For the converse implication, assume that $T$ has the no-collision property, and, by a contradiction, assume that there exist $x_1\neq x_2\in \Omega$ for which the half-space-preserving condition does not hold. Hence, for every $v\in \mathbb{S}^{d-1}$ such that
	\begin{equation*}
	x_2 \cdot v - x_1 \cdot v = 0
	\end{equation*}
	one has that
	\begin{equation*}
	T(x_2) \cdot v - T(x_1) \cdot v = 0.
	\end{equation*}
	Therefore, there exists $\kappa \in \mathbb{R}$ such that
	\begin{equation*}
	T(x_2)-T(x_1)=\kappa (x_2-x_1).
	\end{equation*}
	Next, take $v=\frac{x_2-x_1}{|x_2-x_1|}$. Then, we have that
	\begin{equation*}
	x_2\cdot v-x_1\cdot v >0.
	\end{equation*}
	Furthermore, since $x_1,x_2$ do not satisfy the half-space-preserving condition
	\begin{equation*}
	T(x_2)\cdot v-T(x_1)\cdot v = \kappa |x_2-x_1|< 0,
	\end{equation*}
	and so $\kappa<0$. But then, we have that
	\begin{equation*}
	(1-\lambda)(x_1-x_2)+\lambda (T(x_1)-T(x_2)) =0,
	\end{equation*}
	where
	\begin{equation*}
	\lambda=\frac{\kappa}{\kappa-1} \in (0,1).
	\end{equation*}
	Hence, we arrive at a contradiction with the no-collision property.
\end{proof}

Theorem \ref{thm:no_collision_char} is the starting point for the construction of our transportation maps; that is, we incorporate the half-space-preserving property in the construction. More precisely, assume that $\mu\in \mathcal{P}(\mathbb{R}^d)$ is absolutely continuous and $\{v_k\}$ is an arbitrary sequence of directions. Denote by
\begin{equation*}
\Omega_0=\{\mathbb{R}^d\}.
\end{equation*}
Since $\mu$ is absolutely continuous there exists $h_1\in \mathbb{R}$ such that
\begin{equation*}
\mu(x\cdot v_1 \leq h_1 ) = \mu(x\cdot v_1 > h_1 ).
\end{equation*}
Then we denote by
\begin{equation*}
A_{0}=\{x\cdot v_1 \leq h_1 \},\quad A_{1}= \{x\cdot v_1 > h_1 \},\quad \Omega_1=\{A_{0},A_{1}\}.
\end{equation*}
Furthermore, to every $x\in A$ we assign a binary sequence, $s(x)$, as follows
\begin{equation*}
s(x)=\begin{cases}
0,~x \in A_0,\\
1,~x\in A_1.
\end{cases}
\end{equation*}
Then, we pick a set from $\Omega_1$, say $A_1$, and find $h_2$ such that
\begin{equation*}
\mu(A_1 \cap \{x\cdot v_2 \leq h_2\} ) = \mu(A_1 \cap \{x\cdot v_2 > h_2\}),
\end{equation*}
and denote by
\begin{equation*}
A_{10}=A_1 \cap \{x\cdot v_2 \leq h_2 \},\quad A_{11}=A_1 \cap \{x\cdot v_2 > h_2 \},\quad \Omega_2=\{A_0,A_{10},A_{11}\}.
\end{equation*}
Next, we update $s(x)$ for $x\in A_1$ as follows
\begin{equation*}
s(x)=\begin{cases}
10,~x \in A_{10},\\
11,~x\in A_{11}.
\end{cases}
\end{equation*}
The process continues by picking a subset and dividing it into two pieces of equal masses and updating the corresponding binary sequences. As a result, points in the support of $\mu$ receive a binary sequence $s(x)$. If directions $\{v_k\}$ are chosen appropriately, every two points in $\mathrm{supp}(\mu)$ get separated at some step, and therefore the sequence $s(x)$ is a bijection between $\mathrm{supp}(\mu)$ and $\{0,1\}^\infty$. This procedure, with the same choice of directions and subsets, can be performed for any absolutely continuous measure $\nu$. Suppose that the binary sequence of $\nu$ is $r(x)$. Then, the transportation map from $\mu$ to $\nu$ is
\begin{equation*}
t=r^{-1}\circ s.
\end{equation*}
Although clear intuitively, a rigorous construction of this previous map is a technically challenging task. Denote by
\begin{equation*}
\begin{split}
\hat{s}_0(x)=&0,\quad x\in \mathbb{R}^d,\\
\hat{s}_1(x)=&\begin{cases}
0\cdot\frac{1}{3},~x \in A_0,\\
1\cdot\frac{1}{3},~x\in A_1,
\end{cases}\\
\hat{s}_2(x)=&\begin{cases}
0,~x\in A_0,\\
1\cdot \frac{1}{3}+0\cdot \frac{1}{3^2},~x\in A_{10},\\
1\cdot \frac{1}{3}+1\cdot \frac{1}{3^2},~x\in A_{11},
\end{cases}\\
\vdots
\end{split}
\end{equation*}
So $\hat{s}_k(x),~x\in A$ is the real number in the ternary system with digits $s(x)$.
\begin{lemma}\label{lma:s_k_convergence}
	The family $\{\hat{s}_k\}$ is a bounded, non-decreasing, and uniformly-convergent sequence of measurable step-functions with a limit
	\begin{equation}\label{eq:hat_s}
	\hat{s}(x)=\lim\limits_{k\to \infty} \hat{s}_k(x),\quad x\in \mathbb{R}^d.
	\end{equation}
	Moreover, if for some $k\in \mathbb{N}$ points $x\neq x'$ get separated by a hyperplane parallel to $v_k$, then
	\begin{equation*}
	\hat{s}(x)\neq \hat{s}(x').
	\end{equation*}
\end{lemma}
\begin{proof}
	For all $x,k$ we have that
	\begin{equation*}
	0\leq \hat{s}_k(x)\leq \sum_{i=1}^\infty \frac{1}{3^i}=\frac{1}{2}.
	\end{equation*}
	By construction, the value $\hat{s}_k(x)$ at step $k+1$ can either stay the same or increase by a power of $\frac{1}{3}$. Hence,
	\begin{equation*}
	0 \leq \hat{s}_{k+1}(x)-\hat{s}_k(x) \leq \frac{C}{k},~\forall x\in A,
	\end{equation*}
	for some universal constant $C$.

	Furthermore, suppose that $x\neq x'$ and they get separated at step $k$ for the first time, and denote by $l$ the length of $s(x)=s(x')$ right before step $k$. Then, we get that
	\begin{equation*}
	|\hat{s}_n(x)-\hat{s}_n(x')| \geq \frac{1}{3^l}-\sum_{i=l+1}^\infty \frac{1}{3^i}=\frac{1}{2\cdot 3^l},
	\end{equation*}
	for all $n\geq k$. Therefore, one has that
	\begin{equation*}
	|\hat{s}(x)-\hat{s}(x')|\geq \frac{1}{2\cdot 3^l}>0.
	\end{equation*}
\end{proof}
For a generic absolutely continuous $\nu \in \mathcal{P}(\mathbb{R}^d)$ we denote by $\Delta_k$ the partition sets at step $k$. Furthermore, we denote by $\{\hat{r}_k\}$ and $\hat{r}$ the corresponding step-functions and the limiting function. We call a pair $(A,B)\in \Omega_k \times \Delta_k$ dual if $\hat{s}_k(x)=\hat{r}_k(y)$ for $x\in A,y \in B$. The following theorem is a technical result necessary for the proof of our main existence theorem.

\begin{theorem}\label{thm:hat(t)}
	Suppose that $\mu,\nu$ are absolutely continuous probability measures, $\{v_k\}_{k=1}^\infty$ are some partitioning directions, and $\hat{s}, \hat{r}$ are functions defined by \eqref{eq:hat_s} for $\mu, \nu$, respectively. Furthermore, define a function
	\begin{equation*}\label{eq:hat_t}
	\hat{t}= \hat{r}^{-1} \circ \hat{s}:\mathbb{R}^d \to \mathbb{R}^d,
	\end{equation*}
	where
	\begin{equation*}
	\mathrm{Dom}(\hat{t}) \subset \{x\in \mathbb{R}^d~\mbox{s.t.}~\hat{r}^{-1}(\hat{s}(x))~\mbox{is a singleton}  \}.
	\end{equation*}
	Then one has that
	\begin{equation}\label{eq:hat_t^-1(B)}
	\hat{t}^{-1}(B)=\mathrm{Dom}(\hat{t}) \cap A,~\forall B \in \bigcup_{k=1}^\infty \Delta_k,
	\end{equation}
	where $A \in \bigcup_{k=1}^\infty \Omega_k$ is the dual set of $B$.

	Consequently, if points in $\mathrm{Im}(\hat{t})$ get eventually separated then $\hat{t}$ is a half-space-preserving map.
\end{theorem}
\begin{proof}
	Suppose that $B \in \Delta_k$, and $k$ is the smallest such number. Furthermore, assume that $A\in \Omega_k$ be the dual set corresponding to $B$. Finally, suppose that $x_0\in A,~y_0 \in B$. We first prove that
	\begin{equation}\label{eq:AB_preimages}
	\begin{split}
	A=&\hat{s}^{-1}\left(\left[\hat{s}_k(x_0),\hat{s}_k(x_0)+\frac{1}{2\cdot 3^l}\right]\right),\\
	B=&\hat{r}^{-1}\left(\left[\hat{r}_k(y_0),\hat{r}_k(y_0)+\frac{1}{2\cdot 3^l}\right]\right),
	\end{split}
	\end{equation}
	where $l$ is the length of the binary sequence assigned to the points of $A$ and $B$ at step $k$. Recall that by construction we have that
	\begin{equation*}
	\hat{s}_k(x)=\hat{r}_k(y),~\forall x\in A,~y\in B,
	\end{equation*}
	and
	\begin{equation*}
	\hat{s}_k(x)\leq \hat{s}_n(x) \leq \hat{s}_k(x)+\sum_{i=l+1}^\infty \frac{1}{3^i}=\hat{s}_k(x)+\frac{1}{2\cdot 3^l}
	\end{equation*}
	for any $n\geq k$ and $x \in A$. Therefore,
	\begin{equation*}
	\hat{s}_k(x)\leq \hat{s}(x) \leq \hat{s}_k(x)+\frac{1}{2\cdot 3^l},~\forall x \in A,
	\end{equation*}
	and hence
	\begin{equation*}
	A\subset \hat{s}^{-1}\left(\left[\hat{s}_k(x_0),\hat{s}_k(x_0)+\frac{1}{2\cdot 3^l}\right]\right).
	\end{equation*}
	Now suppose that $x\notin A$. This means that $x \in A'$ for some $A' \in \Omega_k,~A'\neq A$, and $x,x_0$ got separated at a step before $k$. Therefore, from the proof of Lemma \ref{lma:s_k_convergence} we have that
	\begin{equation*}
	|\hat{s}(x')-\hat{s}(x_0)|\geq \frac{1}{2\cdot 3^{l-1}},
	\end{equation*}
	and so
	\begin{equation*}
	\hat{s}(x') \notin \left[\hat{s}_k(x_0),\hat{s}_k(x_0)+\frac{1}{2\cdot 3^l}\right],
	\end{equation*}
	which means that
	\begin{equation*}
	A\supset \hat{s}^{-1}\left(\left[\hat{s}_k(x_0),\hat{s}_k(x_0)+\frac{1}{2\cdot 3^l}\right]\right).
	\end{equation*}
	Thus, \eqref{eq:AB_preimages} is proven. Furthermore, denote by
	\begin{equation*}
	I= \left[\hat{s}_k(x_0),\hat{s}_k(x_0)+\frac{1}{2\cdot 3^l}\right]=\left[\hat{r}_k(y_0),\hat{r}_k(y_0)+\frac{1}{2\cdot 3^l}\right].
	\end{equation*}
	Then we have shown that
	\begin{equation*}
	A=\hat{s}^{-1}(I),\quad B=\hat{r}^{-1}(I).
	\end{equation*}
	Next, we have that $x\in \hat{t}^{-1}(B)$ if and only if $x\in \mathrm{Dom}(\hat{t})$ and $\hat{t}(x)=\hat{r}^{-1}(\hat{s}(x)) \in B=\hat{r}^{-1}(I)$. Latter is equivalent to $\hat{s}(x) \in I$ or $x\in \hat{s}^{-1}(I)=A$. Hence we arrive at \eqref{eq:hat_t^-1(B)}.

	Finally, assume that $x_1\neq x_2 \in \mathrm{Dom}(\hat{t})$. If $\hat{t}(x_1)=\hat{t}(x_2)$ then we can take $v=\frac{x_2-x_1}{|x_2-x_1|}$ and obtain
	\begin{equation*}
	x_2 \cdot v - x_1 \cdot v = |x_2-x_1|>0,\quad \hat{t}(x_2) \cdot v - \hat{t}(x_1) \cdot v=0.
	\end{equation*}
	If $\hat{t}(x_1)\neq \hat{t}(x_2)$ we have that they get separated at some step. Suppose that one step before separating they both belong to some $B$, and we cut $B$ in direction $v$. Then, we get that $B$ gets partitioned into $B',B''$, where
	\begin{equation*}
	B'=B \cap \{y\cdot v  \leq \eta\},\quad B''=B \cap \{y\cdot v > \eta\},
	\end{equation*}
	for some $\eta \in \mathbb{R}$. Without loss of generality assume that
	\begin{equation*}
	\hat{t}(x_1)\in B',\quad \hat{t}(x_2)\in B''.
	\end{equation*}
	Denote by $A,A',A''$ the dual sets of $B,B',B''$, respectively. Then we have that
	\begin{equation*}
	B'=B \cap \{x\cdot v  \leq h\},\quad A''=B \cap \{x\cdot v > h\}.
	\end{equation*}
	Furthermore, from \eqref{eq:hat_t^-1(B)} we have that
	\begin{equation*}
	x_1\in A',\quad x_2\in A''.
	\end{equation*}
	Thus,
	\begin{equation*}
	x_2\cdot v - x_1 \cdot v>0,\quad \hat{t}(x_2) \cdot v - \hat{t}(x_1) \cdot v>0.
	\end{equation*}
\end{proof}

Now we are in the position to state and prove our main existence result for no-collision transportation maps.
\begin{theorem}\label{thm:half_space_pres_map}
	Suppose that
	\begin{equation*}
	\mu= f dx,\quad \nu = g dx,
	\end{equation*}
	are probability measures with compact supports such that
	\begin{equation*}
	\mu\left(\partial(\mathrm{supp}(\mu)) \right)=\nu\left(\partial(\mathrm{supp}(\nu)) \right)=0.
	\end{equation*}
	Furthermore, assume that
	\begin{equation*}
	\begin{split}
	c\leq f(x) \leq C, ~\mu ~\mbox{a.e.},\quad c\leq g(x) \leq C, ~\nu~\mbox{a.e.},
	\end{split}
	\end{equation*}
	for some constants $c,C>0$.

	Finally, suppose that $\{v_k\}_{k=1}^\infty \subset \{e_i\}_{i=1}^d$, where latter is an arbitrary basis in $\mathbb{R}^d$. Moreover, suppose that each set appearing during the partition gets partitioned in each of $\{e_i\}_{i=1}^d$ directions infinitely many times.

	Then, we have that:

	\noindent1. points inside $\mathrm{int}(\mathrm{supp}(\mu))$ and $\mathrm{int}(\mathrm{supp}(\nu))$ get separated,

	\noindent2. there exists $F_0 \in \mathcal{B}(\mathbb{R}^d)$ such that $\mu(F_0)=0$, and for every $x\in \mathrm{int}(\mathrm{supp}(\mu)) \setminus F_0$ there exits a unique $y \in \mathrm{int}(\mathrm{supp}(\nu))$ such that $\hat{s}(x)=\hat{r}(y)$,

	\noindent3. the map $\hat{t}: \mathrm{int}(\mathrm{supp}(\mu)) \setminus F_0 \to \mathrm{int}(\mathrm{supp}(\nu))$, $x\mapsto \hat{r}^{-1}(\hat{s}(x))$ is Borel measurable and $\hat{t} \sharp \mu = \nu$.
\end{theorem}

Next, we prove a regularity result for these maps.
\begin{theorem}\label{thm:hat(t)_cty}
	Under the assumptions of Theorem \ref{thm:half_space_pres_map} one has that $\hat{t}$ is a.e. continuous.
\end{theorem}
We postpone the proofs of these related and rather technical theorems to the Appendix.

\section{Further properties of half-space-preserving transportation maps}\label{sec:hsp_maps_further}

Here, we discuss further critical properties of half-space-preserving maps constructed in Theorem \ref{thm:half_space_pres_map}.
\begin{proposition}[Synthesis]\label{prp:synthesis}
	Suppose that $\mu,\nu \in \mathcal{P}(\mathbb{R}^d)$, $\{v_k\}_{k=1}^\infty$ are some partitioning directions, and $k \in \mathbb{N}$. Furthermore, assume that for all dual pairs $(A,B) \in \Omega_k\times \Delta_k$ there exists a Borel measurable half-space-preserving map $t_{A,B}:A\setminus F_A \to B$ such that
	\begin{equation*}
	t_{A,B} \sharp (\mu|_{A\setminus F_A})=\nu_B,
	\end{equation*}
	where $F_A$ are some Borel sets such that $\mu(F_A)=0$. Then the map $t:\mathbb{R}^d \setminus \bigcup_{A \in \Omega_k} F_A \to \mathbb{R}^d$ given by
	\begin{equation*}
	t=\sum_{A\in \Omega_k} \mathbf{1}_A t_{A,B},
	\end{equation*}
	is a half-space-preserving map such that
	\begin{equation*}
	t\sharp \mu =\nu.
	\end{equation*}
\end{proposition}
This previous proposition asserts that the calculation of half-space-preserving transportation maps on partition sets are mutually independent. Thus, the global-map calculation is decentralized and stable under local perturbations.
\begin{proposition}[Transitivity]\label{prp:transitivity}
	Suppose that $\mu_1,\mu_2,\cdots,\mu_n \in \mathcal{P}(\mathbb{R}^d)$, and $\{v_k\}_{k=1}^\infty$ are some partitioning directions. Furthermore, assume that
	\begin{equation*}
	t_k:\mathbb{R}^d \setminus F_k \to \mathbb{R}^d,~1\leq k \leq n-1,
	\end{equation*}
	are half-space-preserving maps generated by $\{v_k\}_{k=1}^\infty$ such that
	\begin{equation*}
	\mu_k(F_k)=0,\quad t_k \sharp \mu_k =\mu_{k+1},~1\leq k \leq n-1.
	\end{equation*}
	Then, the mapping
	\begin{equation*}
	t=t_{n-1}\circ t_{n-2}\circ \cdots \circ t_1 : \mathbb{R}^d \setminus F_1 \to \mathbb{R}^d,
	\end{equation*}
	is a half-space-preserving mapping corresponding to directions $\{v_k\}_{k=1}^\infty$, and
	\begin{equation*}
	t\sharp \mu_1 = \mu_n.
	\end{equation*}
\end{proposition}

\section{Numerical Results}\label{sec:numerics}



Our algorithm proceeds by successively partitioning the source and target measures. These partitions generate BSP tree decompositions of these measures: see Figures~\ref{fig:cut-grid} and \ref{fig:cut-gauss} for uniform and Gaussian measures illustrated in Figure~\ref{fig:ellipse}.

The algorithm has a run-time of $O(n \log n)$ \cite{blum'73,musser'97}: the points are divided into halves or cells, along a horizontal cut, and then each half is divided into vertical halves. The algorithm is applied recursively until each cell has one point in it or until a stopping criterion is reached. Thus, the algorithm does not involve explicit optimization.

We compare transportation costs of maps generated by our algorithm with optimal ones for various transportation metrics.  As a reference point, we also compute transportation costs of lexicographical-order-preserving or Knothe-Rosenblatt maps \cite{knothe'57,rosen'52}: these are fair candidates with the same run-time and without any optimization of the transportation cost.

\begin{figure}[ht]
	\centering
    \begin{tabular}[b]{c}
        \subfloat[Quasi-equally spaced source points within an ellipse]{
            \includegraphics[width=2.75 in]{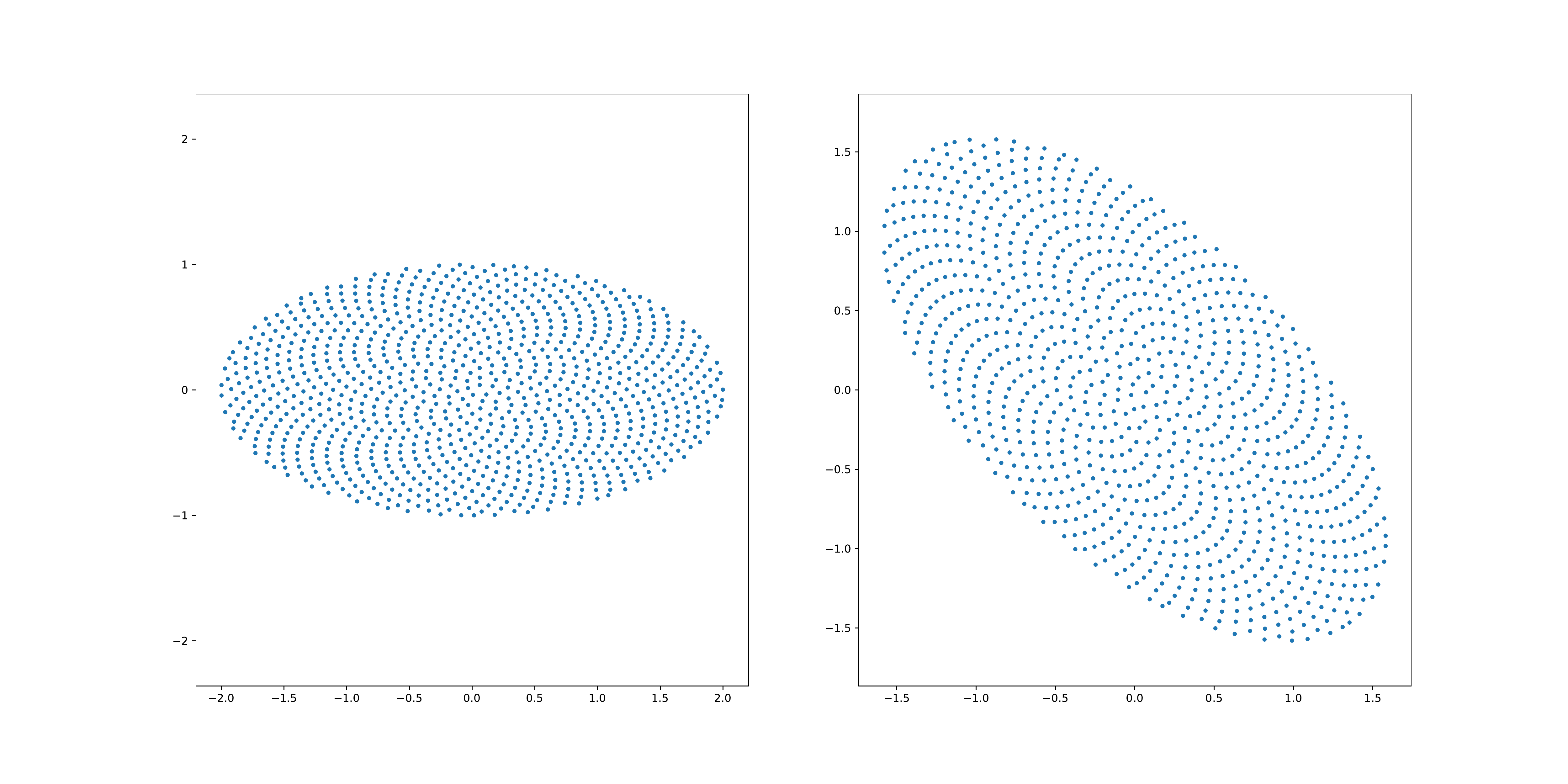}
        }
        \\
        \subfloat[Equally spaced source points within a grid]{
            \includegraphics[width=2.75 in]{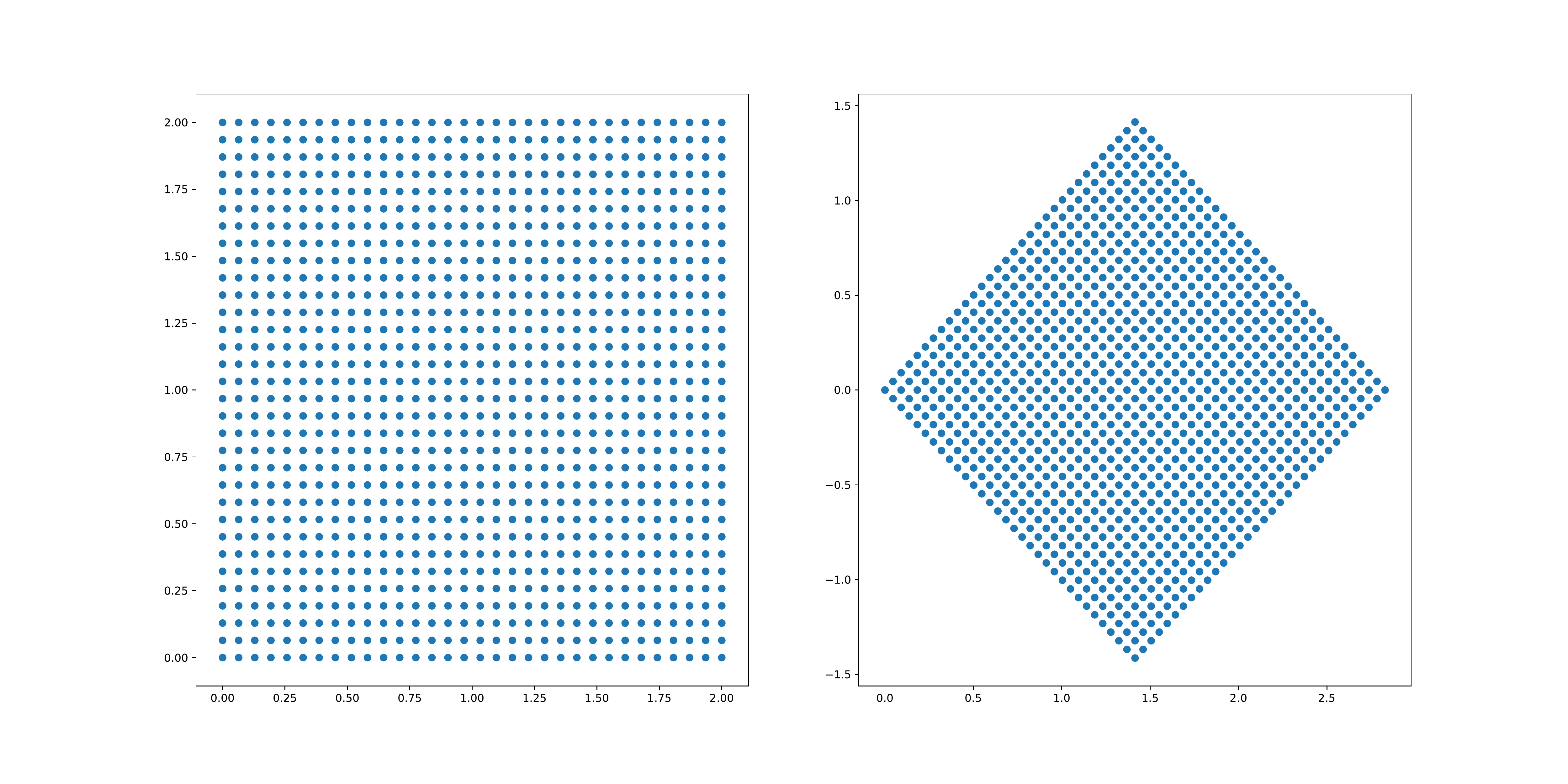}
        }
        \\
        \subfloat[Normally distributed]{
            \includegraphics[width=2.75 in]{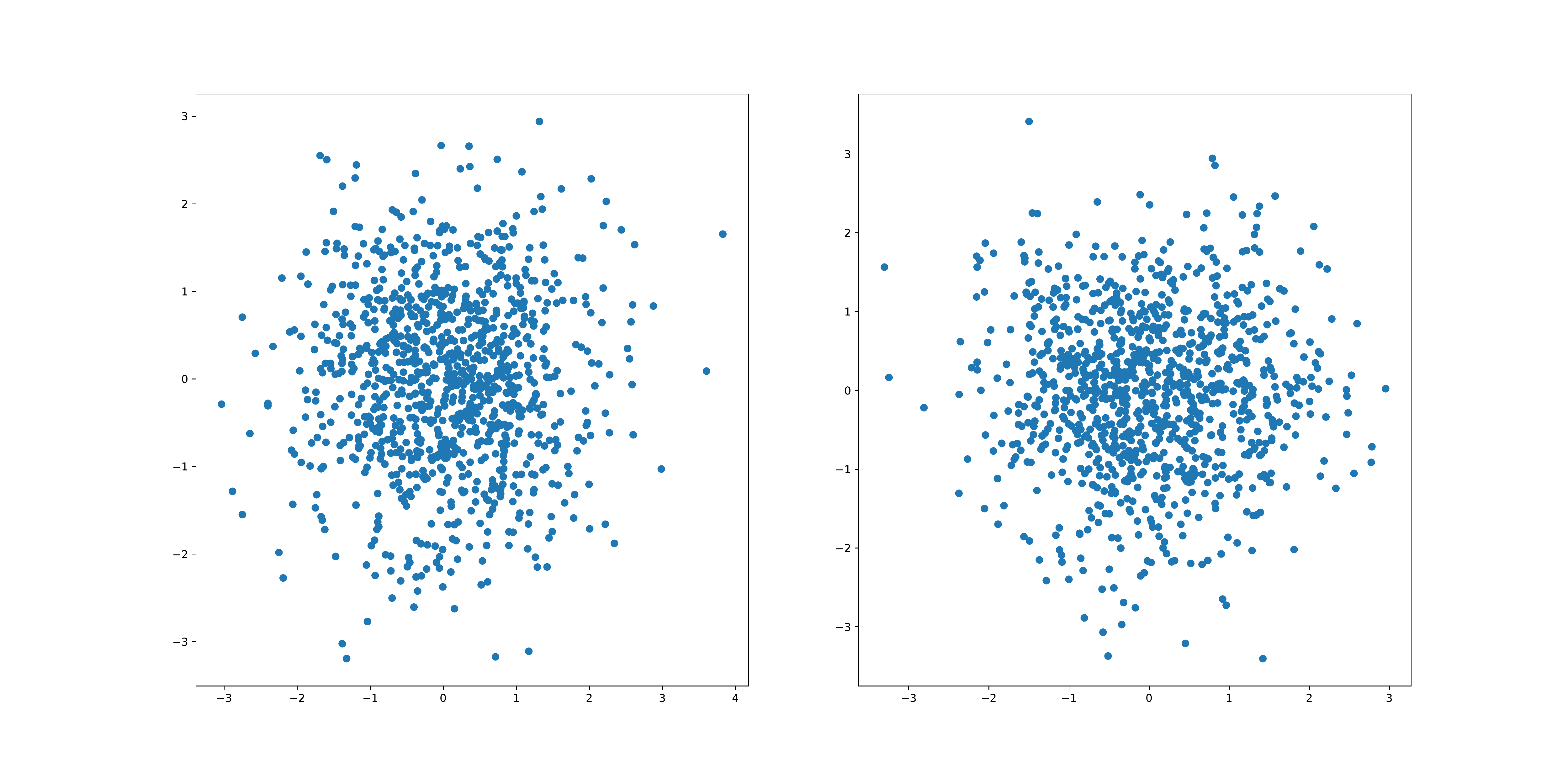}
        }
    \end{tabular}
	\caption{Visualization of computational examples.}
	\label{fig:ellipse}
\end{figure}

\begin{figure}[ht]
    \centering
    \hspace{-.5in}
    \includegraphics[width=2.75 in]{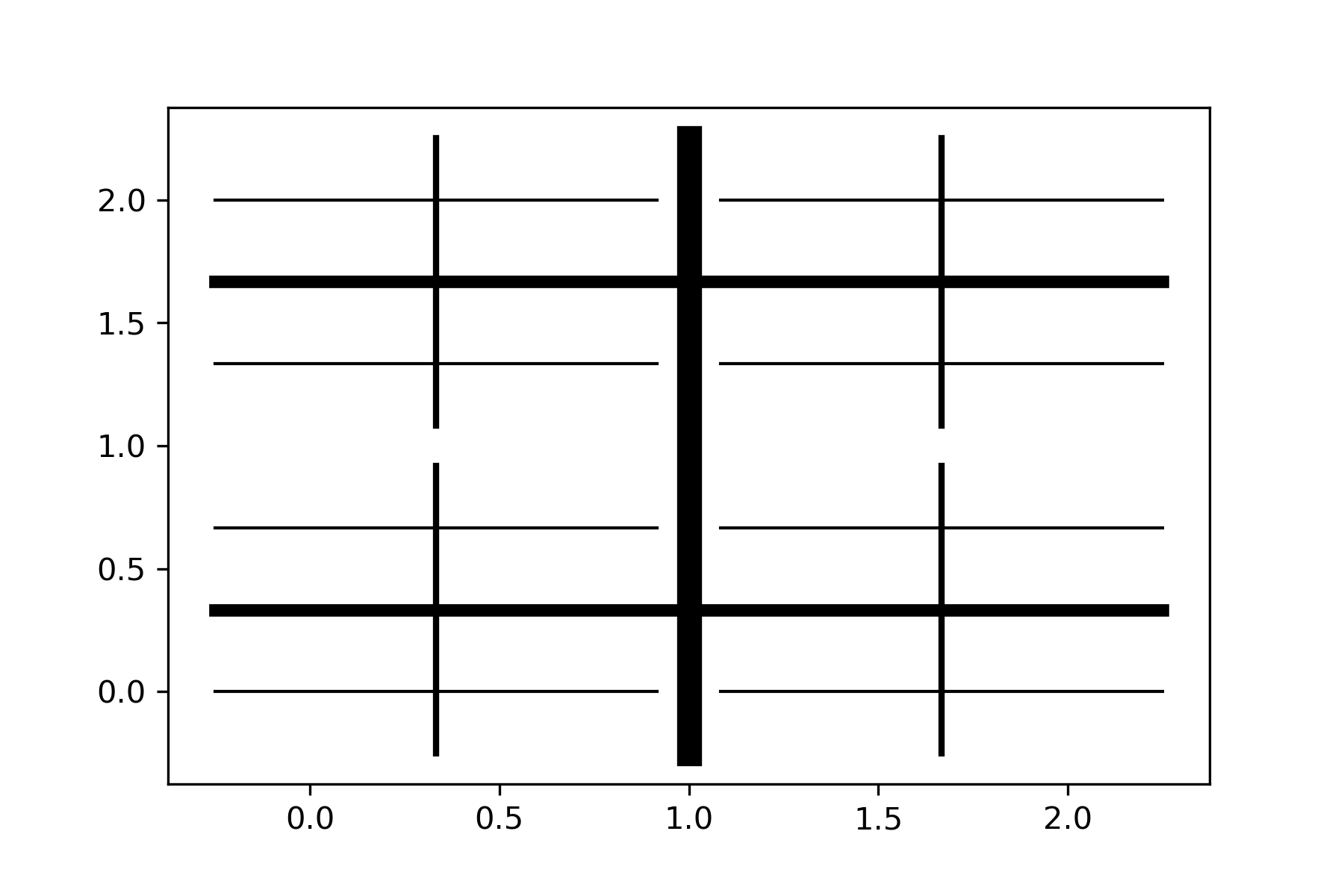}
    \caption{Bisection on equally spaced points within a grid (source)}
    \label{fig:cut-grid}
\end{figure}

\begin{figure}[ht]
    \centering
    \hspace{-.5in}
    \includegraphics[width=2.75 in]{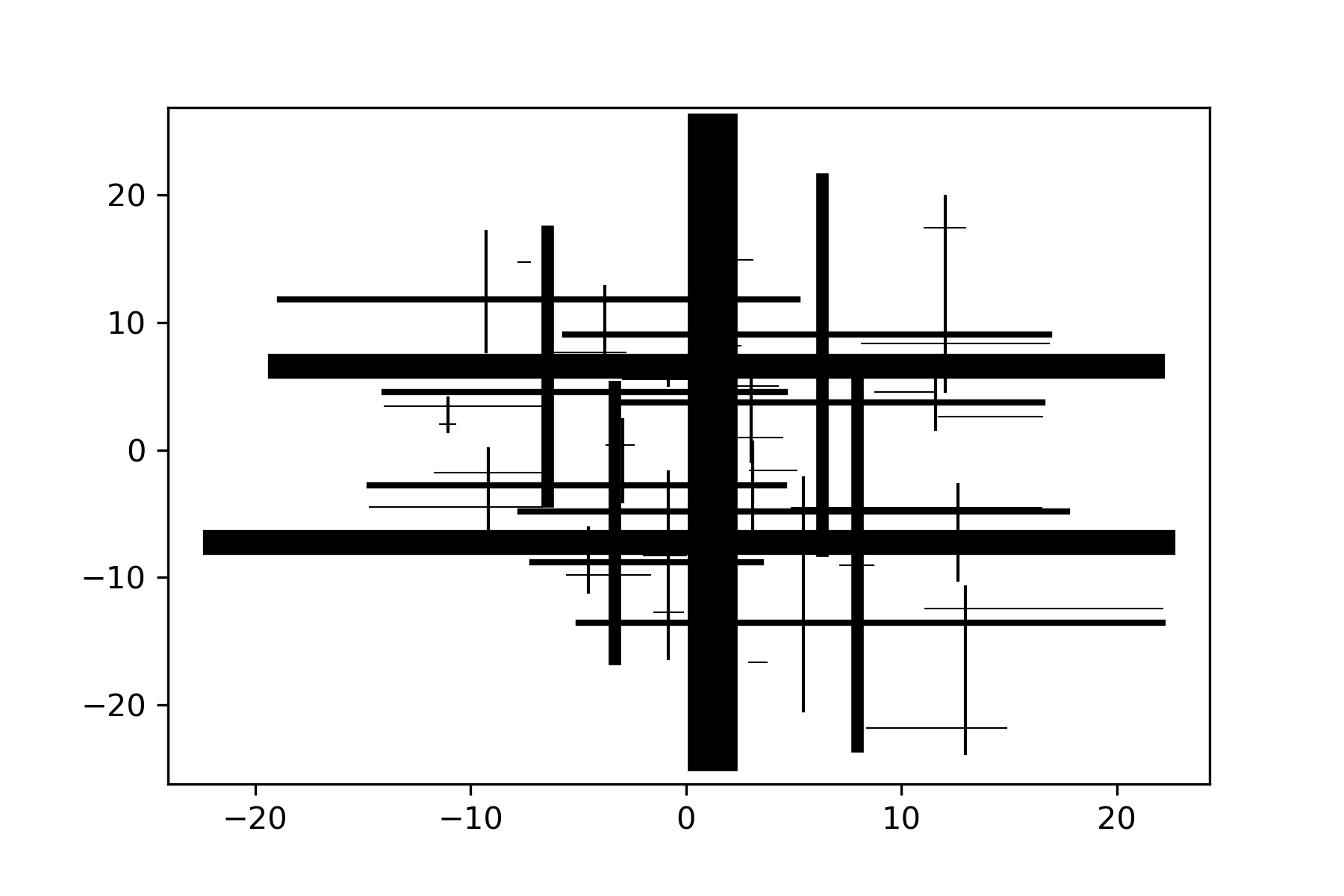}
    \caption{Bisection on points sampled from $\mathcal{N}(0,1)$ (target)}
    \label{fig:cut-gauss}
\end{figure}

\begin{figure}[ht]
    \centering
    \hspace{-.5in}
    \includegraphics[width=2.75 in]{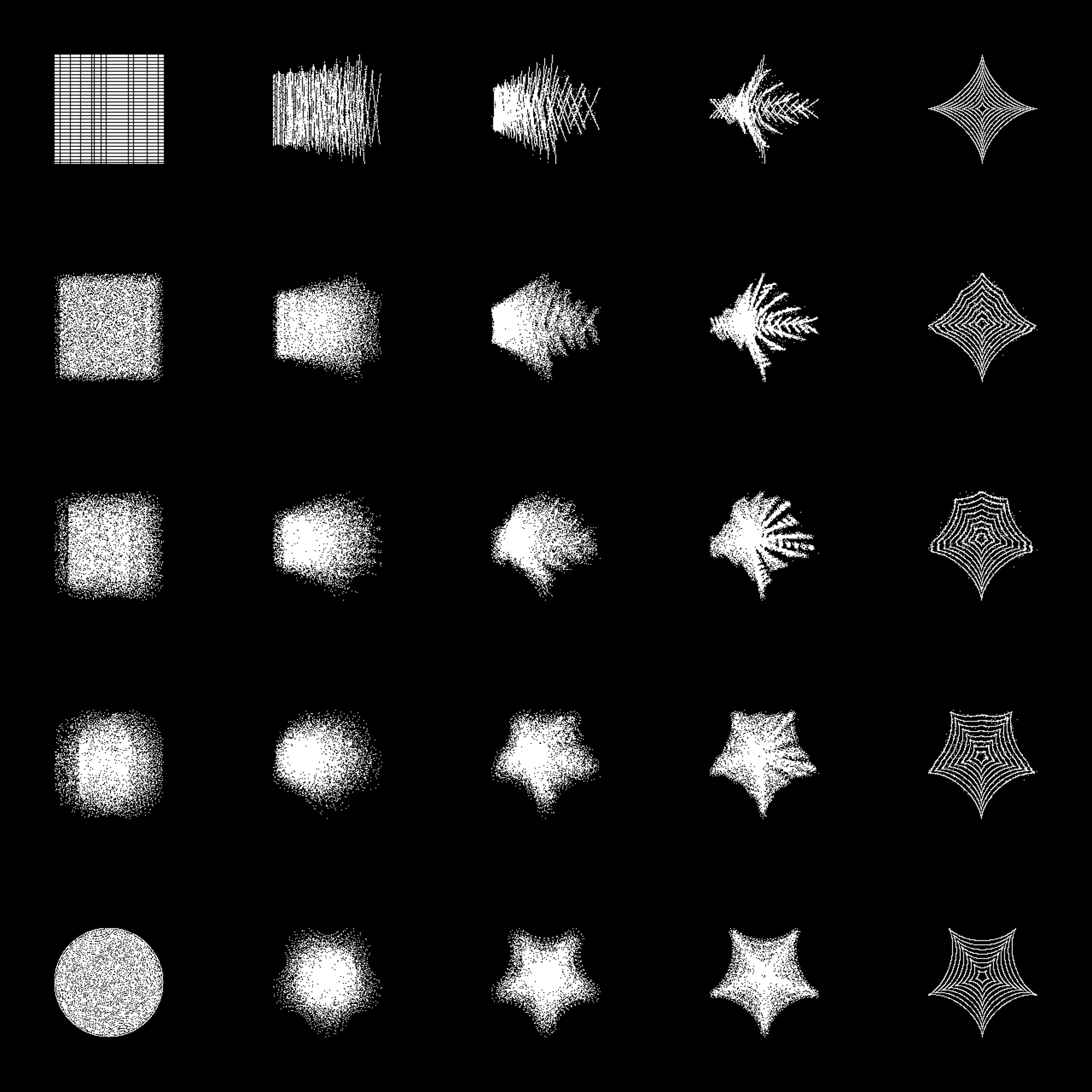}
    \caption{Barycenters of four starting shapes (shown in corners of image) at various weights}
    \label{fig:bary}
\end{figure}




We conducted our experiments on an Intel i7 2.9 GHz CPU with 8 GB of DDR3 RAM for $n = 2^{2N},~N = 2, 3, \ldots, 6$ number of points in the source and target measures.

To compare transportation costs generated by our algorithm, we implemented
\begin{itemize}
\item[(i)] the Hungarian algorithm for the assignment problem that has a run-time of order $n^3$,
\item[(ii)] the linear programming (LP) algorithm for the Kantorovich relaxation of the optimal transportation problem,
\item[(iii)] our algorithm with horizontal-vertical partitioning (HV),
\item[(iv)] the Sinkhorn (SH) algorithm \cite{cuturi2013sinkhorn} for the entropic regularization (ER) of the optimal transportation problem as implemented in \cite{flamary2017pot},
\item[(v)] the lexicographical sorting algorithm (LEX) to generate the Knothe-Rosenblatt map,
\end{itemize}
The Hungarian and LP algorithms yielded the same cost with former breaking down for $N=3$, so we report only the LP cost. In our experiments, we consider transportation costs of the form $c(x,y)=\|x-y\|^q_p$ for $p=1,2,\infty$ and $q=1,2$.

In the first set of examples, we consider empirical measures of quasi-uniformly spaced points in an ellipse, uniformly spaced points on a grid, and normally distributed points with zero mean and an identity covariance matrix, Figure \ref{fig:ellipse}, as source measures. As target measures, we consider their rotated versions by $45^\circ$ counter-clockwise.

%
%
%
%
%
%
%

\begin{table}
	\begin{center}
		\resizebox{0.95\textwidth}{!}{%
			\begin{tabular}{cccccc}
				\toprule
				& & \multicolumn{4}{c}{$n$} \\
				\cmidrule{3-6}
				Cost & Method & $64$ & $256$ & $1024$ & $4096$ \\
				\midrule
				\begin{tabular}{c}
					\multirow{4}{*}{$\|\cdot\|_2^2$} \\
					\phantom{$\|\cdot\|_2^2$} \\
					\phantom{$\|\cdot\|_2^2$} \\
					\phantom{$\|\cdot\|_2^2$}
				\end{tabular}
				&
				\begin{tabular}{cr}
					OT & \\
					HV & (ratio)\\
					SH & (ratio)\\
					LEX & (ratio)
				\end{tabular}
				&
				\begin{tabular}{cr}
                    $0.32$ \\ $0.38$  & ($1.18 \times$) \\ $0.34$ & ($1.07 \times$) \\ $0.85$ & ($2.66 \times$)
				\end{tabular}
				&
				\begin{tabular}{cr}
                    $0.26$ \\ $0.30$  & ($1.16 \times$) \\ $0.30$ & ($1.15 \times$) \\ $0.78$ & ($3.02 \times$)
				\end{tabular}
				&
				\begin{tabular}{cr}
                    $0.24$ \\ $0.28$  & ($1.15 \times$) \\ $0.29$ & ($1.18 \times$) \\ $0.76$ & ($3.09 \times$)
				\end{tabular}
				&
				\begin{tabular}{cr}
					$0.25$ \\ $0.27$  & ($1.08 \times$) \\ $0.29$ & ($1.12 \times$) \\ $0.74$ & ($2.92 \times$)
				\end{tabular}
				\\
				\midrule
				%
				\begin{tabular}{c}
					\multirow{4}{*}{$\|\cdot\|_2$} \\
					\phantom{$\|\cdot\|_2$} \\
					\phantom{$\|\cdot\|_2$} \\
					\phantom{$\|\cdot\|_2$}
				\end{tabular}
				&
				\begin{tabular}{cr}
					OT & \\
					HV & (ratio)\\
					SH & (ratio)\\
					LEX & (ratio)
				\end{tabular}
				&
				\begin{tabular}{cr}
                    $0.49$ \\ $0.58$  & ($1.17 \times$) \\ $0.52$ & ($1.05 \times$) \\ $0.85$ & ($1.72 \times$)
				\end{tabular}
				&
				\begin{tabular}{cr}
                    $0.44$ \\ $0.51$  & ($1.17 \times$) \\ $0.48$ & ($1.08 \times$) \\ $0.81$ & ($1.84 \times$)
				\end{tabular}
				&
				\begin{tabular}{cr}
                    $0.43$ \\ $0.50$  & ($1.18 \times$) \\ $0.47$ & ($1.10 \times$) \\ $0.80$ & ($1.88 \times$)
				\end{tabular}
				&
				\begin{tabular}{cr}
					$0.42$ \\ $0.50$  & ($1.18 \times$) \\ $0.46$ & ($1.11 \times$) \\ $0.79$ & ($1.89 \times$)
				\end{tabular}
				\\
				\midrule
				%
				\begin{tabular}{c}
					\multirow{4}{*}{$\|\cdot\|_1^2$} \\
					\phantom{$\|\cdot\|_1^2$} \\
					\phantom{$\|\cdot\|_1^2$} \\
					\phantom{$\|\cdot\|_1^2$}
				\end{tabular}
				&
				\begin{tabular}{cr}
					OT & \\
					HV & (ratio)\\
					SH & (ratio)\\
					LEX & (ratio)\\
				\end{tabular}
				&
				\begin{tabular}{cr}
                    $0.49$ \\ $0.64$  & ($1.30 \times$) \\ $0.51$ & ($1.04 \times$) \\ $1.40$ & ($2.86 \times$)
				\end{tabular}
				&
				\begin{tabular}{cr}
                    $0.39$ \\ $0.48$  & ($1.23 \times$) \\ $0.42$ & ($1.09 \times$) \\ $1.28$ & ($3.33 \times$)
				\end{tabular}
				&
				\begin{tabular}{cr}
                    $0.36$ \\ $0.43$  & ($1.20 \times$) \\ $0.40$ & ($1.11 \times$) \\ $1.24$ & ($3.45 \times$)
				\end{tabular}
				&
				\begin{tabular}{cr}
					$0.38$ \\ $0.42$  & ($1.11 \times$) \\ $0.40$ & ($1.05 \times$) \\ $1.22$ & ($3.24 \times$)
				\end{tabular}
				\\
				\midrule
				%
				\begin{tabular}{c}
					\multirow{4}{*}{$\|\cdot\|_\infty^2$} \\
					\phantom{$\|\cdot\|_\infty^2$} \\
					\phantom{$\|\cdot\|_\infty^2$} \\
					\phantom{$\|\cdot\|_\infty^2$}
				\end{tabular}
				&
				\begin{tabular}{cr}
					OT & \\
					HV & (ratio)\\
					SH & (ratio)\\
					LEX & (ratio)\\
				\end{tabular}
				&
				\begin{tabular}{cr}
                    $0.23$ \\ $0.30$  & ($1.28 \times$) \\ $0.26$ & ($1.11 \times$) \\ $0.70$ & ($2.97 \times$)
				\end{tabular}
				&
				\begin{tabular}{cr}
                    $0.19$ \\ $0.25$  & ($1.31 \times$) \\ $0.23$ & ($1.20 \times$) \\ $0.64$ & ($3.34 \times$)
				\end{tabular}
				&
				\begin{tabular}{cr}
                    $0.18$ \\ $0.24$  & ($1.32 \times$) \\ $0.22$ & ($1.24 \times$) \\ $0.62$ & ($3.45 \times$)
				\end{tabular}
				&
				\begin{tabular}{cr}
					$0.18$ \\ $0.23$  & ($1.26 \times$) \\ $0.22$ & ($1.20 \times$) \\ $0.61$ & ($3.31 \times$)
				\end{tabular}
				\\
				\bottomrule
		\end{tabular}}
	\end{center}
	\caption{Average costs of set of quasi-equally spaced points within an ellipse to the same set of points rotated $45^\circ$ counter-clockwise measured in various cost functions}
	\label{tab:ellipse2rotellipse}
\end{table}

%
%
%
%
%
%
%

\begin{table}
	\begin{center}
		\resizebox{0.95\textwidth}{!}{%
			\begin{tabular}{cccccc}
				\toprule
				& & \multicolumn{4}{c}{$n$} \\
				\cmidrule{3-6}
				Cost & Method & $64$ & $256$ & $1024$ & $4096$ \\
				\midrule
				\begin{tabular}{c}
					\multirow{4}{*}{$\|\cdot\|_2^2$} \\
					\phantom{$\|\cdot\|_2^2$} \\
					\phantom{$\|\cdot\|_2^2$} \\
					\phantom{$\|\cdot\|_2^2$}
				\end{tabular}
				&
				\begin{tabular}{cr}
					OT & \\
					HV & (ratio)\\
					SH & (ratio)\\
					LEX & (ratio)
				\end{tabular}
				&
				\begin{tabular}{cr}
                    $0.31$ \\ $0.31$  & ($1.01 \times$) \\ $0.35$ & ($1.12 \times$) \\ $0.42$ & ($1.36 \times$)
				\end{tabular}
				&
				\begin{tabular}{cr}
                    $0.30$ \\ $0.31$  & ($1.01 \times$) \\ $0.34$ & ($1.14 \times$) \\ $0.40$ & ($1.33 \times$)
				\end{tabular}
				&
				\begin{tabular}{cr}
                    $0.30$ \\ $0.30$  & ($1.01 \times$) \\ $0.34$ & ($1.14 \times$) \\ $0.40$ & ($1.32 \times$)
				\end{tabular}
				&
				\begin{tabular}{cr}
					$0.31$ \\ $0.30$  & ($0.98 \times$) \\ $0.34$ & ($1.11 \times$) \\ $0.39$ & ($1.27 \times$)
				\end{tabular}
				\\
				\midrule
				%
				\begin{tabular}{c}
					\multirow{4}{*}{$\|\cdot\|_2$} \\
					\phantom{$\|\cdot\|_2$} \\
					\phantom{$\|\cdot\|_2$} \\
					\phantom{$\|\cdot\|_2$}
				\end{tabular}
				&
				\begin{tabular}{cr}
					OT & \\
					HV & (ratio)\\
					SH & (ratio)\\
					LEX & (ratio)
				\end{tabular}
				&
				\begin{tabular}{cr}
                    $0.54$ \\ $0.55$  & ($1.01 \times$) \\ $0.57$ & ($1.05 \times$) \\ $0.60$ & ($1.11 \times$)
				\end{tabular}
				&
				\begin{tabular}{cr}
                    $0.54$ \\ $0.54$  & ($1.01 \times$) \\ $0.57$ & ($1.05 \times$) \\ $0.59$ & ($1.09 \times$)
				\end{tabular}
				&
				\begin{tabular}{cr}
                    $0.54$ \\ $0.54$  & ($1.01 \times$) \\ $0.57$ & ($1.05 \times$) \\ $0.59$ & ($1.09 \times$)
				\end{tabular}
				&
				\begin{tabular}{cr}
					$0.54$ \\ $0.54$  & ($1.01 \times$) \\ $0.57$ & ($1.05 \times$) \\ $0.59$ & ($1.08 \times$)
				\end{tabular}
				\\
				\midrule
				%
				\begin{tabular}{c}
					\multirow{4}{*}{$\|\cdot\|_1^2$} \\
					\phantom{$\|\cdot\|_1^2$} \\
					\phantom{$\|\cdot\|_1^2$} \\
					\phantom{$\|\cdot\|_1^2$}
				\end{tabular}
				&
				\begin{tabular}{cr}
					OT & \\
					HV & (ratio)\\
					SH & (ratio)\\
					LEX & (ratio)\\
				\end{tabular}
				&
				\begin{tabular}{cr}
                    $0.51$ \\ $0.52$  & ($1.02 \times$) \\ $0.53$ & ($1.04 \times$) \\ $0.69$ & ($1.35 \times$)
				\end{tabular}
				&
				\begin{tabular}{cr}
                    $0.51$ \\ $0.52$  & ($1.02 \times$) \\ $0.53$ & ($1.05 \times$) \\ $0.67$ & ($1.32 \times$)
				\end{tabular}
				&
				\begin{tabular}{cr}
                    $0.50$ \\ $0.51$  & ($1.02 \times$) \\ $0.53$ & ($1.05 \times$) \\ $0.65$ & ($1.30 \times$)
				\end{tabular}
				&
				\begin{tabular}{cr}
					$0.50$ \\ $0.51$  & ($1.01 \times$) \\ $0.53$ & ($1.05 \times$) \\ $0.65$ & ($1.29 \times$)
				\end{tabular}
				\\
				\midrule
				%
				\begin{tabular}{c}
					\multirow{4}{*}{$\|\cdot\|_\infty^2$} \\
					\phantom{$\|\cdot\|_\infty^2$} \\
					\phantom{$\|\cdot\|_\infty^2$} \\
					\phantom{$\|\cdot\|_\infty^2$}
				\end{tabular}
				&
				\begin{tabular}{cr}
					OT & \\
					HV & (ratio)\\
					SH & (ratio)\\
					LEX & (ratio)\\
				\end{tabular}
				&
				\begin{tabular}{cr}
                    $0.26$ \\ $0.26$  & ($1.03 \times$) \\ $0.28$ & ($1.09 \times$) \\ $0.35$ & ($1.35 \times$)
				\end{tabular}
				&
				\begin{tabular}{cr}
                    $0.25$ \\ $0.26$  & ($1.02 \times$) \\ $0.28$ & ($1.10 \times$) \\ $0.33$ & ($1.32 \times$)
				\end{tabular}
				&
				\begin{tabular}{cr}
                    $0.25$ \\ $0.26$  & ($1.02 \times$) \\ $0.28$ & ($1.10 \times$) \\ $0.33$ & ($1.30 \times$)
				\end{tabular}
				&
				\begin{tabular}{cr}
					$0.25$ \\ $0.26$  & ($1.02 \times$) \\ $0.28$ & ($1.10 \times$) \\ $0.32$ & ($1.28 \times$)
				\end{tabular}
				\\
				\bottomrule
		\end{tabular}}
	\end{center}
	\caption{Average costs of set of points equally spaced within the unit grid to the same set of points rotated $45^\circ$ counter-clockwise measured in various cost functions}
	\label{tab:grid2rotgrid}
\end{table}

%
%
%
%
%
%
%

\begin{table}
	\begin{center}
		\resizebox{0.95\textwidth}{!}{%
			\begin{tabular}{cccccc}
				\toprule
				& & \multicolumn{4}{c}{$n$} \\
				\cmidrule{3-6}
				Cost & Method & $64$ & $256$ & $1024$ & $4096$ \\
				\midrule
				\begin{tabular}{c}
					\multirow{4}{*}{$\|\cdot\|_2^2$} \\
					\phantom{$\|\cdot\|_2^2$} \\
					\phantom{$\|\cdot\|_2^2$} \\
					\phantom{$\|\cdot\|_2^2$}
				\end{tabular}
				&
				\begin{tabular}{cr}
					OT & \\
					HV & (ratio)\\
					SH & (ratio)\\
					LEX & (ratio)
				\end{tabular}
				&
				\begin{tabular}{cr}
                    $0.23$ \\ $0.49$  & ($2.15 \times$) \\ $0.24$ & ($1.08 \times$) \\ $1.14$ & ($5.05 \times$)
				\end{tabular}
				&
				\begin{tabular}{cr}
                    $0.10$ \\ $0.22$  & ($2.20 \times$) \\ $0.13$ & ($1.30 \times$) \\ $1.14$ & ($11.23 \times$)
				\end{tabular}
				&
				\begin{tabular}{cr}
                    $0.04$ \\ $0.09$  & ($2.49 \times$) \\ $0.08$ & ($2.12 \times$) \\ $1.20$ & ($32.91 \times$)
				\end{tabular}
				&
				\begin{tabular}{cr}
					$0.01$ \\ $0.04$  & ($3.08 \times$) \\ $0.06$ & ($4.55 \times$) \\ $1.18$ & ($91.63 \times$)
				\end{tabular}
				\\
				\midrule
				%
				\begin{tabular}{c}
					\multirow{4}{*}{$\|\cdot\|_2$} \\
					\phantom{$\|\cdot\|_2$} \\
					\phantom{$\|\cdot\|_2$} \\
					\phantom{$\|\cdot\|_2$}
				\end{tabular}
				&
				\begin{tabular}{cr}
					OT & \\
					HV & (ratio)\\
					SH & (ratio)\\
					LEX & (ratio)
				\end{tabular}
				&
				\begin{tabular}{cr}
                    $0.39$ \\ $0.49$  & ($1.25 \times$) \\ $0.41$ & ($1.06 \times$) \\ $0.86$ & ($2.19 \times$)
				\end{tabular}
				&
				\begin{tabular}{cr}
                    $0.19$ \\ $0.32$  & ($1.63 \times$) \\ $0.22$ & ($1.15 \times$) \\ $0.86$ & ($4.44 \times$)
				\end{tabular}
				&
				\begin{tabular}{cr}
                    $0.15$ \\ $0.24$  & ($1.64 \times$) \\ $0.19$ & ($1.28 \times$) \\ $0.97$ & ($6.68 \times$)
				\end{tabular}
				&
				\begin{tabular}{cr}
					$0.07$ \\ $0.13$  & ($1.69 \times$) \\ $0.13$ & ($1.76 \times$) \\ $0.95$ & ($12.85 \times$)
				\end{tabular}
				\\
				\midrule
				%
				\begin{tabular}{c}
					\multirow{4}{*}{$\|\cdot\|_1^2$} \\
					\phantom{$\|\cdot\|_1^2$} \\
					\phantom{$\|\cdot\|_1^2$} \\
					\phantom{$\|\cdot\|_1^2$}
				\end{tabular}
				&
				\begin{tabular}{cr}
					OT & \\
					HV & (ratio)\\
					SH & (ratio)\\
					LEX & (ratio)\\
				\end{tabular}
				&
				\begin{tabular}{cr}
                    $0.41$ \\ $0.74$  & ($1.78 \times$) \\ $0.43$ & ($1.05 \times$) \\ $1.60$ & ($3.86 \times$)
				\end{tabular}
				&
				\begin{tabular}{cr}
                    $0.14$ \\ $0.27$  & ($1.94 \times$) \\ $0.16$ & ($1.19 \times$) \\ $1.98$ & ($14.39 \times$)
				\end{tabular}
				&
				\begin{tabular}{cr}
                    $0.06$ \\ $0.16$  & ($2.49 \times$) \\ $0.10$ & ($1.61 \times$) \\ $1.92$ & ($30.63 \times$)
				\end{tabular}
				&
				\begin{tabular}{cr}
					$0.02$ \\ $0.05$  & ($3.06 \times$) \\ $0.06$ & ($3.80 \times$) \\ $1.86$ & ($117.01 \times$)
				\end{tabular}
				\\
				\midrule
				%
				\begin{tabular}{c}
					\multirow{4}{*}{$\|\cdot\|_\infty^2$} \\
					\phantom{$\|\cdot\|_\infty^2$} \\
					\phantom{$\|\cdot\|_\infty^2$} \\
					\phantom{$\|\cdot\|_\infty^2$}
				\end{tabular}
				&
				\begin{tabular}{cr}
					OT & \\
					HV & (ratio)\\
					SH & (ratio)\\
					LEX & (ratio)\\
				\end{tabular}
				&
				\begin{tabular}{cr}
                    $0.17$ \\ $0.31$  & ($1.81 \times$) \\ $0.19$ & ($1.13 \times$) \\ $0.97$ & ($5.66 \times$)
				\end{tabular}
				&
				\begin{tabular}{cr}
                    $0.09$ \\ $0.21$  & ($2.33 \times$) \\ $0.12$ & ($1.37 \times$) \\ $0.96$ & ($10.92 \times$)
				\end{tabular}
				&
				\begin{tabular}{cr}
                    $0.03$ \\ $0.08$  & ($3.28 \times$) \\ $0.07$ & ($2.64 \times$) \\ $0.97$ & ($38.32 \times$)
				\end{tabular}
				&
				\begin{tabular}{cr}
					$0.01$ \\ $0.04$  & ($3.95 \times$) \\ $0.06$ & ($5.77 \times$) \\ $0.96$ & ($99.29 \times$)
				\end{tabular}
				\\
				\bottomrule
		\end{tabular}}
	\end{center}
	\caption{Average costs of sets of points sampled from $\mathcal{N}(0,1)$ to the same set of points rotated $45^\circ$ counter-clockwise measured in various cost functions}
	\label{tab:gauss2rotgauss}
\end{table}

We summarize the results in Tables \ref{tab:ellipse2rotellipse}, \ref{tab:grid2rotgrid}, and \ref{tab:gauss2rotgauss}. Note that only the Hungarian, HV, and LEX algorithms return a map. In practice, however, the LP solution yields a plan that is nearly a map. The SH solution is a plan that can have large support, depending on the ER regularization parameter. We found that the algorithm was faster with a regularization parameter 0.01 that led to quite wide plans. We report the costs of maps or plans as appropriate. The conversion of plans to maps imposes extra costs associated with the projection that we also include.


From Tables \ref{tab:ellipse2rotellipse}, \ref{tab:grid2rotgrid}, and \ref{tab:gauss2rotgauss} we find that the accuracy of HV costs are comparable to those of SH. However, HV algorithm yields a map directly. Furthermore, HV maps perform much better than Knothe-Rosenblatt maps.

In the second set of examples, we map the uniform distribution on the unit square to the standard Gaussian, Table \ref{tab:grid2gauss}, and the latter to a rotated Gaussian with a $3:1$ aspect ration, Table \ref{tab:gauss2newgauss}. We again observe that HV maps consistently estimate the optimal cost within a factor of 2 and in some cases, achieve remarkable accuracy of a few percent.

Finally, in Figure \ref{fig:bary}, we use HV maps to construct barycenters or displacement interpolations of various shapes. The results clearly indicate that HV maps are shape-sensitive and encode geometric information.

\begin{table}
\begin{center}
\resizebox{0.95\textwidth}{!}{%
\begin{tabular}{cccccc}
\toprule
& & \multicolumn{4}{c}{$n$} \\
\cmidrule{3-6}
Cost & Method & $64$ & $256$ & $1024$ & $4096$ \\
\midrule
\begin{tabular}{c}
    \multirow{4}{*}{$\|\cdot\|_2^2$} \\
    \phantom{$\|\cdot\|_2^2$} \\
    \phantom{$\|\cdot\|_2^2$} \\
    \phantom{$\|\cdot\|_2^2$}
\end{tabular}
&
\begin{tabular}{cr}
    OT & \\
    HV & (ratio)\\
    SH & (ratio)\\
    LEX & (ratio)
\end{tabular}
&
\begin{tabular}{cr}
    $1.56$ \\ $1.58$  & ($1.01 \times$) \\ $1.60$ & ($1.02 \times$) \\ $2.86$ & ($1.83 \times$)
\end{tabular}
&
\begin{tabular}{cr}
    $1.76$ \\ $1.76$  & ($1.00 \times$) \\ $1.80$ & ($1.02 \times$) \\ $3.02$ & ($1.72 \times$)
\end{tabular}
&
\begin{tabular}{cr}
    $1.53$ \\ $1.53$  & ($1.00 \times$) \\ $1.57$ & ($1.03 \times$) \\ $2.65$ & ($1.73 \times$)
\end{tabular}
&
\begin{tabular}{cr}
    $1.55$ \\ $1.54$  & ($0.99 \times$) \\ $1.59$ & ($1.02 \times$) \\ $2.68$ & ($1.73 \times$)
\end{tabular}
\\
\midrule
%
\begin{tabular}{c}
    \multirow{4}{*}{$\|\cdot\|_2$} \\
    \phantom{$\|\cdot\|_2$} \\
    \phantom{$\|\cdot\|_2$} \\
    \phantom{$\|\cdot\|_2$}
\end{tabular}
&
\begin{tabular}{cr}
    OT & \\
    HV & (ratio)\\
    SH & (ratio)\\
    LEX & (ratio)
\end{tabular}
&
\begin{tabular}{cr}
    $0.96$ \\ $0.98$  & ($1.01 \times$) \\ $0.99$ & ($1.03 \times$) \\ $1.42$ & ($1.47 \times$)
\end{tabular}
&
\begin{tabular}{cr}
    $1.07$ \\ $1.08$  & ($1.01 \times$) \\ $1.10$ & ($1.03 \times$) \\ $1.46$ & ($1.36 \times$)
\end{tabular}
&
\begin{tabular}{cr}
    $1.05$ \\ $1.05$  & ($1.00 \times$) \\ $1.08$ & ($1.03 \times$) \\ $1.42$ & ($1.35 \times$)
\end{tabular}
&
\begin{tabular}{cr}
   $1.09$ \\ $1.09$  & ($1.00 \times$) \\ $1.12$ & ($1.03 \times$) \\ $1.45$ & ($1.33 \times$)
\end{tabular}
\\
\midrule
%
\begin{tabular}{c}
    \multirow{4}{*}{$\|\cdot\|_1^2$} \\
    \phantom{$\|\cdot\|_1^2$} \\
    \phantom{$\|\cdot\|_1^2$} \\
    \phantom{$\|\cdot\|_1^2$}
\end{tabular}
&
\begin{tabular}{cr}
    OT & \\
    HV & (ratio)\\
    SH & (ratio)\\
    LEX & (ratio)\\
\end{tabular}
&
\begin{tabular}{cr}
    $3.23$ \\ $3.31$  & ($1.03 \times$) \\ $3.25$ & ($1.01 \times$) \\ $5.44$ & ($1.69 \times$)
\end{tabular}
&
\begin{tabular}{cr}
    $2.42$ \\ $2.46$  & ($1.02 \times$) \\ $2.45$ & ($1.01 \times$) \\ $4.26$ & ($1.76 \times$)
\end{tabular}
&
\begin{tabular}{cr}
    $2.53$ \\ $2.56$  & ($1.01 \times$) \\ $2.56$ & ($1.01 \times$) \\ $4.43$ & ($1.75 \times$)
\end{tabular}
&
\begin{tabular}{cr}
    $2.52$ \\ $2.53$  & ($1.01 \times$) \\ $2.54$ & ($1.01 \times$) \\ $4.41$ & ($1.75 \times$)
\end{tabular}
\\
\midrule
%
\begin{tabular}{c}
    \multirow{4}{*}{$\|\cdot\|_\infty^2$} \\
    \phantom{$\|\cdot\|_\infty^2$} \\
    \phantom{$\|\cdot\|_\infty^2$} \\
    \phantom{$\|\cdot\|_\infty^2$}
\end{tabular}
&
\begin{tabular}{cr}
    OT & \\
    HV & (ratio)\\
    SH & (ratio)\\
    LEX & (ratio)\\
\end{tabular}
&
\begin{tabular}{cr}
    $0.97$ \\ $1.01$  & ($1.04 \times$) \\ $0.99$ & ($1.02 \times$) \\ $2.15$ & ($2.22 \times$)
\end{tabular}
&
\begin{tabular}{cr}
    $1.34$ \\ $1.35$  & ($1.01 \times$) \\ $1.36$ & ($1.02 \times$) \\ $2.27$ & ($1.70 \times$)
\end{tabular}
&
\begin{tabular}{cr}
    $1.25$ \\ $1.25$  & ($1.00 \times$) \\ $1.28$ & ($1.02 \times$) \\ $2.16$ & ($1.73 \times$)
\end{tabular}
&
\begin{tabular}{cr}
    $1.28$ \\ $1.29$  & ($1.00 \times$) \\ $1.31$ & ($1.02 \times$) \\ $2.23$ & ($1.74 \times$)
\end{tabular}
\\
\bottomrule
\end{tabular}}
\end{center}
\caption{Average costs of set of points equally spaced within a grid to a set of points sampled from $\mathcal{N}(0,1)$ measured in various cost functions}
\label{tab:grid2gauss}
\end{table}

\begin{table}
\begin{center}
\resizebox{0.95\textwidth}{!}{%
\begin{tabular}{cccccc}
\toprule
& & \multicolumn{4}{c}{$n$} \\
\cmidrule{3-6}
Cost & Method & $64$ & $256$ & $1024$ & $4096$ \\
\midrule
\begin{tabular}{c}
    \multirow{4}{*}{$\|\cdot\|_2^2$} \\
    \phantom{$\|\cdot\|_2^2$} \\
    \phantom{$\|\cdot\|_2^2$} \\
    \phantom{$\|\cdot\|_2^2$}
\end{tabular}
&
\begin{tabular}{cr}
    OT & \\
    HV & (ratio)\\
    SH & (ratio)\\
    LEX & (ratio)
\end{tabular}
&
\begin{tabular}{cr}
    $0.97$ \\ $1.49$  & ($1.54 \times$) \\ $0.98$ & ($1.02 \times$) \\ $6.52$ & ($6.72 \times$)
\end{tabular}
&
\begin{tabular}{cr}
    $1.02$ \\ $1.23$  & ($1.21 \times$) \\ $1.05$ & ($1.03 \times$) \\ $6.59$ & ($6.46 \times$)
\end{tabular}
&
\begin{tabular}{cr}
    $1.23$ \\ $1.32$  & ($1.07 \times$) \\ $1.26$ & ($1.03 \times$) \\ $7.00$ & ($5.72 \times$)
\end{tabular}
&
\begin{tabular}{cr}
    $1.14$ \\ $1.16$  & ($1.02 \times$) \\ $1.16$ & ($1.02 \times$) \\ $7.13$ & ($6.27 \times$)
\end{tabular}
\\
\midrule
%
\begin{tabular}{c}
    \multirow{4}{*}{$\|\cdot\|_2$} \\
    \phantom{$\|\cdot\|_2$} \\
    \phantom{$\|\cdot\|_2$} \\
    \phantom{$\|\cdot\|_2$}
\end{tabular}
&
\begin{tabular}{cr}
    OT & \\
    HV & (ratio)\\
    SH & (ratio)\\
    LEX & (ratio)
\end{tabular}
&
\begin{tabular}{cr}
    $1.01$ \\ $1.15$  & ($1.13 \times$) \\ $1.04$ & ($1.02 \times$) \\ $2.25$ & ($2.22 \times$)
\end{tabular}
&
\begin{tabular}{cr}
    $0.77$ \\ $0.83$  & ($1.08 \times$) \\ $0.80$ & ($1.03 \times$) \\ $2.18$ & ($2.83 \times$)
\end{tabular}
&
\begin{tabular}{cr}
    $0.78$ \\ $0.81$  & ($1.04 \times$) \\ $0.81$ & ($1.04 \times$) \\ $2.27$ & ($2.92 \times$)
\end{tabular}
&
\begin{tabular}{cr}
    $0.84$ \\ $0.85$  & ($1.02 \times$) \\ $0.87$ & ($1.04 \times$) \\ $2.31$ & ($2.75 \times$)
\end{tabular}
\\
\midrule
%
\begin{tabular}{c}
    \multirow{4}{*}{$\|\cdot\|_1^2$} \\
    \phantom{$\|\cdot\|_1^2$} \\
    \phantom{$\|\cdot\|_1^2$} \\
    \phantom{$\|\cdot\|_1^2$}
\end{tabular}
&
\begin{tabular}{cr}
    OT & \\
    HV & (ratio)\\
    SH & (ratio)\\
    LEX & (ratio)\\
\end{tabular}
&
\begin{tabular}{cr}
    $1.27$ \\ $1.74$  & ($1.38 \times$) \\ $1.28$ & ($1.01 \times$) \\ $9.42$ & ($7.44 \times$)
\end{tabular}
&
\begin{tabular}{cr}
    $1.47$ \\ $1.75$  & ($1.19 \times$) \\ $1.49$ & ($1.01 \times$) \\ $10.37$ & ($7.07 \times$)
\end{tabular}
&
\begin{tabular}{cr}
    $1.22$ \\ $1.37$  & ($1.13 \times$) \\ $1.25$ & ($1.03 \times$) \\ $11.27$ & ($9.27 \times$)
\end{tabular}
&
\begin{tabular}{cr}
    $1.09$ \\ $1.15$  & ($1.05 \times$) \\ $1.10$ & ($1.01 \times$) \\ $11.13$ & ($10.21 \times$)
\end{tabular}
\\
\midrule
%
\begin{tabular}{c}
    \multirow{4}{*}{$\|\cdot\|_\infty^2$} \\
    \phantom{$\|\cdot\|_\infty^2$} \\
    \phantom{$\|\cdot\|_\infty^2$} \\
    \phantom{$\|\cdot\|_\infty^2$}
\end{tabular}
&
\begin{tabular}{cr}
    OT & \\
    HV & (ratio)\\
    SH & (ratio)\\
    LEX & (ratio)\\
\end{tabular}
&
\begin{tabular}{cr}
    $1.39$ \\ $1.79$  & ($1.29 \times$) \\ $1.41$ & ($1.01 \times$) \\ $6.58$ & ($4.75 \times$)
\end{tabular}
&
\begin{tabular}{cr}
    $1.23$ \\ $1.49$  & ($1.21 \times$) \\ $1.26$ & ($1.02 \times$) \\ $5.97$ & ($4.85 \times$)
\end{tabular}
&
\begin{tabular}{cr}
    $1.06$ \\ $1.19$  & ($1.12 \times$) \\ $1.08$ & ($1.03 \times$) \\ $5.89$ & ($5.57 \times$)
\end{tabular}
&
\begin{tabular}{cr}
    $1.00$ \\ $1.05$  & ($1.05 \times$) \\ $1.03$ & ($1.02 \times$) \\ $5.97$ & ($5.94 \times$)
\end{tabular}
\\
\bottomrule
\end{tabular}}
\end{center}
\caption{Average costs of sets of points sampled from $\mathcal{N}(0,1)$ to a new set of points sampled from $\mathcal{N}(0,1)$ (then scaled to a 3:1 aspect ratio and rotated $90^\circ$ counter-clockwise) measured in various cost functions}
\label{tab:gauss2newgauss}
\end{table}


%
%

\section*{Appendix}

\begin{proof}[Theorem \ref{thm:half_space_pres_map}]
We assume that $\{e_i\}_{i=1}^d$ is the standard basis in $\mathbb{R}^d$ because the proof for a general basis is identical up to a multiplication by a suitable volume element.

\noindent1. Suppose that $x \in \mathrm{int}\left(\mathrm{supp}(\mu)\right)$, and $~x'\in \mathbb{R}^d,~x\neq x',$ but they never get separated by a hyperplane. Therefore, whenever a common subset containing $x,x'$ gets partitioned they always stay in the same side. Suppose that $\{A_k\}$ is the sequence of subsets that they both belong during the cutting process. By construction, we have that
\begin{equation*}
A_{k+1}\subset A_k,\quad \mu(A_{k+1})=\frac{\mu(A_k)}{2}.
\end{equation*}
Since $\{v_k\}_{k=1}^\infty \subset \{e_1,e_2,\cdots,e_d\}$ we have that
\begin{equation*}
A_k=(\alpha^k_1,\beta_1^k] \times (\alpha^k_2,\beta^k_2]\cdots \times (\alpha^k_d,\beta^k_d],
\end{equation*}
where $-\infty \leq \alpha^k_i <\beta^k _i \leq +\infty$. Denote by
\begin{equation*}
x=(x_1,x_2,\cdots,x_d),\quad x'=(x'_1,x'_2,\cdots,x'_d).
\end{equation*}
Without loss of generality, assume that
\begin{equation*}
x_i\neq x'_i,~1\leq i \leq l,\quad x_i=x'_i,~i>l.
\end{equation*}
Since $\{A_k\}$ are rectangles we have that
\begin{equation*}
R=\bigcap_k A_k,
\end{equation*}
is also a rectangle, and we denote by
\begin{equation*}
R=[\alpha_1,\beta_1]\times [\alpha_2,\beta_2]\times \cdots [\alpha_d,\beta_d],
\end{equation*}
and we have that
\begin{equation*}
\lim\limits_{k\to \infty}\alpha_i^k=\alpha_i,\quad \lim\limits_{k\to \infty}\beta_i^k=\beta_i,~1\leq i \leq d.
\end{equation*}
Since $x,x' \in R$ we have that
\begin{equation*}
\beta_i-\alpha_i \geq |x_i-x'_i|>0,~1\leq i \leq l.
\end{equation*}
Furthermore, we have that
\begin{equation*}
\mu(R)=\lim\limits_{k\to \infty} \mu(A_k)=0.
\end{equation*}
If $\mathcal{L}^d(R)>0$ then we have that $\mathrm{int}(R)\neq \emptyset$ and $\mathrm{int}(R)\cap \mathrm{supp}(\mu)=\emptyset$ which contradicts to the fact that $x \in \mathrm{int}\left(\mathrm{supp}(\mu)\right)$. Therefore, we have that $\mathcal{L}^d(R)=0$ which means that $\alpha_i=\beta_i$ for some $i>l$. Without loss of generality assume that
\begin{equation*}
\alpha_i<\beta_i,~1\leq i \leq q,\quad \alpha_i=\beta_i,~i>q.
\end{equation*}
We have that $q\geq l$. Moreover, $\alpha_i=\beta_i=x_i=x'_i$, and $-\infty<\alpha_i^k<\beta_i^k<\infty$ for all $i>q$ and $k$ large enough. Additionally, if $-\infty<\alpha_i<\beta_i<\infty$ for some $1\leq i \leq q$ then $-\infty<\alpha_i^k<\beta_i^k<\infty$ for $k$ large enough. In what follows we assume that $k$ is so large that this previous statements hold.

Furthermore, assume that $M>0$ is such that
\begin{equation*}
\mathrm{supp}(\mu) \subset [-M,M]^d.
\end{equation*}
Since $\mu=fdx$, by construction we have that
\begin{equation*}
\int_{A_k\setminus A_{k+1}} fdx= \int_{A_{k+1}} fdx,~\forall k.
\end{equation*}
Therefore, using $c\leq f \leq C,~\mu$ a.e. we get that
\begin{equation}\label{eq:c/C_estimate}
\frac{\mathcal{L}^d(A_k\setminus A_{k+1} \cap \mathrm{supp}(\mu))}{\mathcal{L}^d(A_{k+1}\cap \mathrm{supp}(\mu) ) } \geq \frac{c}{C}>0,~\forall k.
\end{equation}
Now, suppose that for some $k$ the set $A_k$ gets partitioned in the direction $e_1$. There are three possibilities: a) $-\infty<\alpha_1<\beta_1<\infty$, b) $-\infty<\alpha_1<\beta_1=\infty$, c) $-\infty=\alpha_1<\beta_1<\infty$.

\noindent a) $-\infty<\alpha_1<\beta_1<\infty$. In this case, we have that $-\infty<\alpha_1^k<\beta_1^k<\infty$ since $k$ is large enough.Therefore, either
\begin{equation*}
\begin{split}
A_{k+1}=(\alpha^k_1,\gamma] \times (\alpha^k_2,\beta^k_2]\cdots \times (\alpha^k_d,\beta^k_d],\\
A_k\setminus A_{k+1}=(\gamma,\beta^k_1] \times (\alpha^k_2,\beta^k_2]\cdots \times (\alpha^k_d,\beta^k_d],
\end{split}
\end{equation*}
or
\begin{equation*}
\begin{split}
A_{k+1}=(\gamma,\beta^k_1] \times (\alpha^k_2,\beta^k_2]\cdots \times (\alpha^k_d,\beta^k_d],\\
A_k\setminus A_{k+1}=(\alpha^k_1,\gamma] \times (\alpha^k_2,\beta^k_2]\cdots \times (\alpha^k_d,\beta^k_d],
\end{split}
\end{equation*}
for some $\alpha_1^k<\gamma<\beta_1^k$. Suppose that we are in the former case.

Since $x\in \mathrm{int}(\mathrm{supp}(\mu))$ we have that there exists a $\sigma>0$ such that
\begin{equation*}
\times_{i=1}^d [x_i-\sigma,x_i+\sigma] \subset \mathrm{supp}(\mu).
\end{equation*}
We have that
\begin{equation*}
A_k\setminus A_{k+1} \cap \mathrm{supp}(\mu) \subset A_k\setminus A_{k+1}\cap [-M,M]^d,
\end{equation*}
and therefore
\begin{equation*}
\begin{split}
&\mathcal{L}^d(A_k\setminus A_{k+1} \cap \mathrm{supp}(\mu)) \\
\leq& \mathcal{L}^d(A_k\setminus A_{k+1}\cap [-M,M]^d)\leq  (\beta_1^k-\gamma) \prod_{i=2}^d \min \{\beta_i^k-\alpha_i^k,2M\}\\
\leq &(\beta_1^k-\beta_1) (2M)^{q-1} \prod_{i>q} (\beta_i^k-\alpha_i^k),
\end{split}
\end{equation*}
where we used the fact that $\beta_1\leq \gamma < \beta_1^k$ and
\begin{equation*}
\lim\limits_{k\to \infty} \alpha_i^k= \lim\limits_{k\to \infty} \beta_i^k=x_i,~i>q.
\end{equation*}
On the other hand, we have that
\begin{equation*}
A_{k+1} \cap \mathrm{supp}(\mu) \supset A_{k+1} \cap \times_{i=1}^d [x_i-\sigma,x_i+\sigma],
\end{equation*}
therefore
\begin{equation*}
\begin{split}
&\mathcal{L}^d(A_{k+1} \cap \mathrm{supp}(\mu))\\
\geq & \mathcal{L}^d(A_{k+1} \cap \times_{i=1}^d [x_i-\sigma,x_i+\sigma])\geq \min\{\gamma-\alpha_1^k,\sigma\} \prod_{i=2}^d \min\{ \beta_i^k-\alpha_i^k,\sigma  \}\\
\geq & \prod_{i=1}^q \min\{ \beta_i-\alpha_i,\sigma  \} \prod_{i>q} (\beta_i^k-\alpha_i^k).
\end{split}
\end{equation*}
Hence, we obtain that
\begin{equation*}
\frac{\mathcal{L}^d(A_k\setminus A_{k+1} \cap \mathrm{supp}(\mu))}{\mathcal{L}^d(A_{k+1} \cap \mathrm{supp}(\mu))} \leq (\beta_1^k-\beta_1) \frac{(2M)^{q-1}}{\prod_{i=1}^q \min\{ \beta_i-\alpha_i,\sigma  \}}.
\end{equation*}
Similarly, if $x,x'$ fall in the upper (in $e_1$ direction) half of $A_k$ we get that
\begin{equation*}
\frac{\mathcal{L}^d(A_k\setminus A_{k+1} \cap \mathrm{supp}(\mu))}{\mathcal{L}^d(A_{k+1} \cap \mathrm{supp}(\mu))} \leq (\alpha_1-\alpha_1^k) \frac{(2M)^{q-1}}{\prod_{i=1}^q \min\{ \beta_i-\alpha_i,\sigma  \}}.
\end{equation*}

\noindent b) $-\infty<\alpha_1<\beta_1=\infty$. In this case we have that $-\infty<\alpha_1^k<\beta_1^k=\infty$, and
\begin{equation*}
\begin{split}
A_{k+1}=(\gamma,\infty] \times (\alpha^k_2,\beta^k_2]\cdots \times (\alpha^k_d,\beta^k_d],\\
A_k\setminus A_{k+1}=(\alpha^k_1,\gamma] \times (\alpha^k_2,\beta^k_2]\cdots \times (\alpha^k_d,\beta^k_d],
\end{split}
\end{equation*}
for some $\alpha_1^k<\gamma<\infty$. As before, we have that
\begin{equation*}
\begin{split}
& \mathcal{L}^d(A_k\setminus A_{k+1} \cap \mathrm{supp}(\mu))\\
& \leq (\gamma-\alpha_1^k) (2M)^{q-1} \prod_{i>q}(\beta_i^k-\alpha_i^k) \leq (\alpha_1-\alpha_1^k) (2M)^{q-1} \prod_{i>q}(\beta_i^k-\alpha_i^k).
\end{split}
\end{equation*}
Similarly, we have that
\begin{equation*}
\mathcal{L}^d(A_{k+1} \cap \mathrm{supp}(\mu)) \geq \prod_{i=1}^q \min\{ \beta_i-\alpha_i,\sigma  \} \prod_{i>q} (\beta_i^k-\alpha_i^k),
\end{equation*}
and thus
\begin{equation*}
\frac{\mathcal{L}^d(A_k\setminus A_{k+1} \cap \mathrm{supp}(\mu))}{\mathcal{L}^d(A_{k+1} \cap \mathrm{supp}(\mu))} \leq (\alpha_1-\alpha_1^k) \frac{(2M)^{q-1}}{\prod_{i=1}^q \min\{ \beta_i-\alpha_i,\sigma  \}}.
\end{equation*}

\noindent c) $-\infty=\alpha_1<\beta_1<\infty$. In this case, we have that $-\infty=\alpha_1^k<\beta_1^k<\infty$, and
\begin{equation*}
\frac{\mathcal{L}^d(A_k\setminus A_{k+1} \cap \mathrm{supp}(\mu))}{\mathcal{L}^d(A_{k+1} \cap \mathrm{supp}(\mu))} \leq (\beta_1^k-\beta_1) \frac{(2M)^{q-1}}{\prod_{i=1}^q \min\{ \beta_i-\alpha_i,\sigma  \}}.
\end{equation*}

Summarizing, we get whenever $A_k$ gets partitioned in $e_1$ we have that
\begin{equation*}
\frac{\mathcal{L}^d(A_k\setminus A_{k+1} \cap \mathrm{supp}(\mu))}{\mathcal{L}^d(A_{k+1} \cap \mathrm{supp}(\mu))} \leq o(1).
\end{equation*}
We get similar estimates for partitions in any of the directions $\{e_i\}_{i=1}^q$. Thus, if we take a subsequence $\{A_{k_m}\}$ that get partitioned in one of these directions we get that
\begin{equation*}
\lim\limits_{m\to \infty} \frac{\mathcal{L}^d(A_{k_m}\setminus A_{k_m+1} \cap \mathrm{supp}(\mu))}{\mathcal{L}^d(A_{k_m+1} \cap \mathrm{supp}(\mu))} =0,
\end{equation*}
which contradicts to \eqref{eq:c/C_estimate}. Thus, the first item is proven.

\noindent2. Firstly, we will show that for every $x\in \mathrm{int}(\mathrm{supp}(\mu))$ there exists $y \in \mathrm{supp}(\nu)$ such that
\begin{equation*}
\hat{s}(x)=\hat{r}(y).
\end{equation*}
Assume that $\{A_k\}$ is the sequence of partition sets that contain $x$. Again, we have that
\begin{equation*}
A_{k+1}\subset A_k,\quad \mu(A_{k+1})=\frac{\mu(A_k)}{2},
\end{equation*}
and
\begin{equation*}
A_k=(\alpha^k_1,\beta_1^k] \times (\alpha^k_2,\beta^k_2]\cdots \times (\alpha^k_d,\beta^k_d],
\end{equation*}
for some $-\infty \leq \alpha^k_i <\beta^k _i \leq +\infty$. Moreover, from the previous item  we obtain  that
\begin{equation*}
\bigcap_k A_k=\{x\},
\end{equation*}
because $x\in \mathrm{int}(\mathrm{supp}(\mu))$, and the intersection cannot contain any other point. Therefore, we have that $-\infty < \alpha^k_i <\beta^k _i < +\infty$, and
\begin{equation*}
\alpha^k_i \nearrow x_i,\quad \beta^k_i \searrow x_i,~ \mbox{as}~k\to \infty,
\end{equation*}
where $x=(x_1,x_2,\cdots,x_d)$.

Denote by $\{B_k\}$ the dual sequence of $\{A_k\}$ that partition $\nu$. Again, we have that
\begin{equation*}
B_k=(\gamma^k_1,\delta_1^k] \times (\gamma^k_2,\delta^k_2]\cdots \times (\gamma^k_d,\delta^k_d],
\end{equation*}
for some $-\infty < \gamma^k_i <\delta^k _i < +\infty$. Thus, our first task is to show that
\begin{equation*}
\bigcap_k B_k \cap \mathrm{supp}(\nu)\neq \emptyset.
\end{equation*}

Since $\alpha_i^k<x_i$ we have that $\{\alpha_i^k\}_k$ is not an eventually constant sequence. Therefore, $\{\gamma^k_i\}_k$ is also not an eventually constant sequence. Besides, $\{\gamma^k_i\}_k$ and $\{\delta^k_i\}_k$ are, respectively, nondecreasing and nonincreasing sequences. Therefore, we have that
\begin{equation*}
W_x=\bigcap_k B_k =[\gamma_1,\delta_1]\times [\gamma_2,\delta_2] \cdots \times [\gamma_d,\delta_d],
\end{equation*}
where
\begin{equation*}
\gamma_i=\sup_k \gamma_i^k,\quad \delta_i=\inf_k \delta_i^k,~1\leq i \leq d.
\end{equation*}
In fact, we have that
\begin{equation*}
W_x=\bigcap_k \mathrm{cl}(B_k).
\end{equation*}
Since $\nu(B_k)>0$ we have that $\mathrm{cl}(B_k) \cap \mathrm{supp}(\nu) \neq \emptyset$. Thus,
\begin{equation*}
\{\mathrm{cl}(B_k) \cap \mathrm{supp}(\nu)\}_k
\end{equation*}
is a nested family of nonempty compact sets. Therefore, we have that
\begin{equation*}
W_x\cap \mathrm{supp}(\nu)= \bigcap_k \mathrm{cl}(B_k) \cap \mathrm{supp}(\nu) \neq \emptyset.
\end{equation*}
If $W_x\cap \mathrm{int}(\mathrm{supp}(\nu)) \neq \emptyset$ then by item 1, we get that
\begin{equation*}
W_x=\{y\},
\end{equation*}
for some $y \in \mathrm{int}(\mathrm{supp}(\nu))$. Hence, to complete the proof of item 1, we need to show that there exists a $F_0 \in \mathcal{B}(\mathbb{R}^d)$ such that $\mu(F_0)=0$, and
\begin{equation*}
W_x \cap \mathrm{int}(\mathrm{supp}(\nu)) \neq \emptyset,\quad \forall x\notin \mathrm{int}(\mathrm{supp}(\mu))\setminus F_0.
\end{equation*}
For every $k$ denote by
\begin{equation*}
\Delta'_k=\{B \in \Delta_k ~\mbox{s,t.}~B \cap \partial(\mathrm{supp}(\nu)) \neq \emptyset   \},
\end{equation*}
and
\begin{equation*}
\Delta''_k=\Delta_k \setminus \Delta'_k.
\end{equation*}
Since
\begin{equation*}
\bigcup_{\Delta_k} B =\mathbb{R}^d,
\end{equation*}
we obtain that
\begin{equation*}
\partial(\mathrm{supp}(\nu) ) \subset \bigcup_{\Delta'_k} B=H_k.
\end{equation*}
Furthermore, for every $B \in \Delta_k''$ we have that $\nu(B)>0$, and therefore $B \cap \mathrm{supp}(\nu) \neq \emptyset$. On the other hand, $B \cap \partial(\mathrm{supp}(\nu)) = \emptyset$, and $B$ is connected. Hence, $B \subset \mathrm{int}(\mathrm{supp}(\nu))$, and
\begin{equation*}
G_k= \bigcup_{\Delta''_k} B \subset \mathrm{int}(\mathrm{supp}(\nu)).
\end{equation*}
Note that
\begin{equation*}
H_k \supset H_{k+1},\quad G_k \subset G_{k+1},~\forall k.
\end{equation*}
By item 1 we have that for every $y \in \mathrm{int}(\mathrm{supp}(\nu))$ there exists a partition rectangle $B$ such that $y\in B \subset \mathrm{int}(\mathrm{supp}(\nu))$. Hence, $B \in \Delta''_k$ for some $k$, and $y \in G_k$. Therefore, we obtain that
\begin{equation*}
\bigcup_k G_k =\mathrm{int}(\mathrm{supp}(\nu)).
\end{equation*}
Consequently,
\begin{equation*}
\bigcap_k H_k =\mathbb{R}^d \setminus \mathrm{int}(\mathrm{supp}(\nu)) \supset \partial(\mathrm{supp}(\nu)).
\end{equation*}
Since $\nu(\mathrm{int}(\mathrm{supp}(\nu)))=1$ we get that
\begin{equation*}
\nu\left( \bigcup_k G_k \right)=1,\quad \nu\left( \bigcap_k H_k \right)=0.
\end{equation*}
Denote by $\Omega_k'$ and $\Omega_k''$ the families dual to $\Delta_k'$ and $\Delta_k''$. Furthermore, denote by
\begin{equation*}
F_k= \bigcup_{\Omega_k'} A,\quad E_k= \bigcup_{\Omega_k''} A.
\end{equation*}
By construction, we have that
\begin{equation*}
F_k \supset F_{k+1},\quad E_k \subset E_{k+1},~\forall k.
\end{equation*}
Moreover,
\begin{equation*}
\mu(F_k)=\nu(H_k),\quad \mu(E_k)=\nu(G_k).
\end{equation*}
Denote by
\begin{equation*}
F_0=\bigcap_k F_k.
\end{equation*}
Then, we have that $F_0 \in \mathcal{B}(\mathbb{R}^d)$, and
\begin{equation*}
\mu(F_0)=\lim\limits_{k\to \infty} \mu(F_k)=\lim\limits_{k\to \infty} \nu(H_k)=\nu\left( \bigcap_k H_k \right)=0.
\end{equation*}
Finally, note that if $W_x \cap \mathrm{int}(\mathrm{supp}(\nu))=\emptyset$ then $x\in F_0$.

\noindent 3. From items 1, 2 we have that the map $\hat{t}: \mathrm{int}(\mathrm{supp}(\mu)) \setminus F_0 \to \mathrm{int}(\mathrm{supp}(\nu))$ given by
\begin{equation*}
\hat{t}(x)=\hat{r}^{-1}(\hat{s}(x)),
\end{equation*}
is well defined. Our first task is to show that $\hat{t}$ is Borel measurable. For that, we need to show that $\hat{t}^{-1}(G) \in \mathcal{B}(\mathbb{R}^d)$ for any open set $G \subset \mathbb{R}^d$. Since $\mathrm{Im}(\hat{t}) \subset \mathrm{int}(\mathrm{supp}(\nu))$ we have that
\begin{equation*}
\hat{t}^{-1}(G)= \hat{t}^{-1}\left(G \cap \mathrm{int}(\mathrm{supp}(\nu)) \right).
\end{equation*}
Therefore, we may assume that $G \subset \mathrm{int}(\mathrm{supp}(\nu))$. Furthermore, denote by
\begin{equation*}
\Delta_k''=\left\{ B \in \Delta_k~\mbox{s.t.}~B\subset G   \right\},\quad \Delta_k'=\Delta_k \setminus \Delta_k''.
\end{equation*}
Next, define
\begin{equation*}
G_k=\bigcup_{\Delta''_k} B.
\end{equation*}
Then we have that
\begin{equation*}
G_{k}\subset G_{k+1},\quad G_k \subset G,~\forall k.
\end{equation*}
From item 1, we have that for every $y\in G$ there exists a partition set $B$ such that $y\in B \subset G$. Hence, we get that
\begin{equation*}
\bigcup_k G_k=G.
\end{equation*}
This means that
\begin{equation*}
\hat{t}^{-1}(G)=\bigcup_k \hat{t}^{-1}(G_k)=\bigcup_k \bigcup_{\Delta''_k} \hat{t}^{-1}(B).
\end{equation*}
On the other hand, from equation \eqref{eq:hat_t^-1(B)} in Theorem \ref{thm:hat(t)} we have that
\begin{equation*}
\hat{t}^{-1}(G)=\bigcup_k \bigcup_{\Omega''_k} (A \cap \mathrm{int}(\mathrm{supp}(\mu)) \setminus F_0),
\end{equation*}
where $\Omega_k',\Omega_k''$ are the dual families of $\Delta_k',\Delta_k''$. Therefore, we have that $\hat{t}^{-1}(G) \in \mathcal{B}(\mathbb{R}^d)$. Furthermore, denote by
\begin{equation*}
F_k= \bigcup_{\Omega''_k} (A \setminus F_0),~\forall k.
\end{equation*}
By construction, we have that
\begin{equation*}
F_k \subset F_{k+1},~\forall k,
\end{equation*}
and hence
\begin{equation*}
\mu\left(\hat{t}^{-1}(G)\right)= \lim\limits_{k\to \infty} \mu(F_k).
\end{equation*}
On the other hand, from construction and equations $\mu(F_0)=0$, $\mu(\mathrm{int}(\mathrm{supp}(\mu)))=1$, we have that
\begin{equation*}
\mu(F_k)=\sum_{\Omega_k''} \mu(A)= \sum_{\Delta_k''} \nu(B)=\nu(G_k).
\end{equation*}
Therefore, we obtain that
\begin{equation*}
\mu\left(\hat{t}^{-1}(G)\right)= \lim\limits_{k\to \infty} \mu(F_k)= \lim\limits_{k\to \infty} \nu(G_k)=\nu(G).
\end{equation*}
Thus, $\hat{t}$ is Borel measurable and $\hat{t}\sharp \mu=\nu$. Finally, from item 1 we have that points in $\mathrm{Im}(\hat{t})\subset \mathrm{int}(\mathrm{supp}(\nu))$ get eventually separated. Therefore, by Theorem \ref{thm:hat(t)} we obtain that $\hat{t}$ is half-space-preserving.
\end{proof}

\begin{proof}[Theorem \ref{thm:hat(t)_cty}]
Denote by $H$ the union of all the hyperplanes that partition $\mu$. Then we have that $\mathcal{L}^d(H)=\mu(H)=0$. We prove that $\hat{t}$ is continuous on $\mathrm{Dom}(\hat{t})\setminus H$.

Suppose $x\in \mathrm{Dom}(\hat{t})\setminus H$, and $y=\hat{t}(x)$. Furthermore, denote by $\{A_k\}$ and $\{B_k\}$ the rectangles that contain $x$ and $y$, respectively. We have that
\begin{equation*}
\bigcap_k A_k=\{x\},\quad  \bigcap_k B_k=\{y\}.
\end{equation*}
Moreover, since $x\notin H$ we have that $x\in \mathrm{int}(A_k)$ for all $k$. Therefore, by construction, we have that $y\in \mathrm{int}(B_k)$ for all $k$. Thus, we obtain that
\begin{equation*}
\bigcap_k \mathrm{int}(A_k)=\{x\},\quad  \bigcap_k \mathrm{int}(B_k)=\{y\},
\end{equation*}
which yields the continuity.
\end{proof}


\bibliographystyle{spmpsci}

\end{document}